\documentclass[12pt,draft,reqno]{amsart}

\usepackage{amssymb} 
\usepackage{dsfont} 
\usepackage{mathrsfs} 
\usepackage{stmaryrd} 
\usepackage{color}
\usepackage{xcolor}

\usepackage{soul}

\DeclareMathAlphabet{\mathpzc}{OT1}{pzc}{m}{it} 

\newtheorem{Thm}{Theorem}[section]

\newtheorem{Lem}{Lemma}[section]
\newtheorem{Prop}{Proposition}[section]

\newtheorem{Def}{Definition}[section]

\theoremstyle{definition} 
\newtheorem{Ass}{Assumptions}[section]

\theoremstyle{definition} 

\theoremstyle{definition}

\theoremstyle{definition}

\makeatletter
\@addtoreset{equation}{section}
\makeatother

\def\fdual#1#2{{_{V'}\langle #1,#2\rangle_V}}

\def\distrdual#1#2{{_{\mathscr{D}'(\Omega)}\langle #1,#2\rangle_{\mathscr{D}(\Omega)}}}

\def\thetadiff{{\sl\Theta}}
\def\chidiff{{\mathcal X}}
\def\thetadiffhat{\widehat{{\sl\Theta}}}

\newcommand{\thetabarep}{\overline{\theta}_\eps}

\newcommand\vuoto{\varnothing} 
\newcommand\parti[1]{\mathscr{P}(#1)} 

\newcommand{\eps}{\varepsilon} 

\DeclareMathOperator{\de}{d \! \hspace{0.04ex}} 
\newcommand\intom{\int_\Omega} 
\newcommand\integr{\int_0^t} 
\newcommand\normaq[2]{\Vert #1\Vert_{#2}^2} 
\newcommand\thetaep{\theta_{\eps}}
\newcommand\chiep{\chi_{\eps}}

\newcommand\xiep{\xi_{\eps}}
\newcommand\thetahatep{\widehat{\theta}_{\eps}}
\newcommand\fhat{\widehat{f}}
\newcommand\fep{f_\eps}

\newcommand\thetazep{\theta_{0\eps}}
\newcommand\chizep{\chi_{0\eps}}

\newcommand\xizep{\xi_{0\eps}}

\newcommand\norma[2]{\Vert #1\Vert_{#2}} 

\def\fdual#1#2{{_{V'}\langle #1,#2\rangle_V}}

\newcommand{\convergedebstar}{\stackrel{*}{\rightharpoonup}}

\newcommand{\Thetatilde}{\widetilde{\Theta}}

\newcommand\thetahat{\widehat{\theta}}

\renewcommand\L[1]{L^#1(\Omega)} 
\renewcommand\H[1]{H^#1(\Omega)} 

\newcommand\sign{\mathrm{sign}} 

\newcommand\erre{\mathbb{R}} 

\renewcommand\div{\mathrm {div\hspace{0.5ex}}} 

\newcommand\vett[1]{\mathbf{#1}} 

\newcommand{\thetatilde}{\widetilde{\theta}}
\newcommand{\chitilde}{\widetilde{\chi}}

\newcommand\function{\longrightarrow} 

\newcommand\en{\mathbb{N}} 

\providecommand{\clint}[1]{\hspace{0.045ex}\left[#1\right]} 

\newcommand\norm[2]{\Vert #1\Vert_{#2}} 

\providecommand{\clsxint}[1]{\hspace{0.1ex}\left[#1\right[\hspace{0.15ex}} 
\providecommand{\cldxint}[1]{\hspace{0.15ex}\left]#1\right]} 
\providecommand{\opint}[1]{\hspace{0.15ex}\left]#1\right[\hspace{0.15ex}} 

\newcommand\setmeno{\!\smallsetminus\!} 

\renewcommand{\L}{{\textsl{L}\hspace{0.17ex}}} 

\newcommand{\convergedeb}{\rightharpoonup} 

\definecolor{blu}{rgb}{0.1,0.1,1}

\definecolor{green}{rgb}{0.0, 0.5, 0.0}

\definecolor{marr}{rgb}{0.63, 0.47, 0.35}

\definecolor{magenta(dye)}{rgb}{0.79, 0.08, 0.48} %

\definecolor{violet}{rgb}{0.54,0.17,0.89}

\begin{document}


\title[Stefan problem]{Asymptotic analysis for the Stefan problem with a doubly nonlinear phase relaxation}

\author{Vincenzo Recupero, Chiara Zanini}
\thanks{The first author is a member of GNAMPA-INdAM}
        
\address{\textbf{Vincenzo Recupero} \\
        Dipartimento di Scienze Matematiche \\ 
        Politecnico di Torino \\
        C.so Duca degli Abruzzi 24 \\ 
        I-10129 Torino \\ 
        Italy. \newline
        {\rm E-mail address:}
        {\tt vincenzo.recupero@polito.it}}
\address{\textbf{Chiara Zanini} \\
        Dipartimento di Scienze Matematiche \\ 
        Politecnico di Torino \\
        C.so Duca degli Abruzzi 24 \\ 
        I-10129 Torino \\ 
        Italy. \newline
        {\rm E-mail address:}
        {\tt chiara.zanini@polito.it}}




\begin{abstract}
In this paper we consider the Stefan problem endowed with a doubly nonlinear phase relaxation. More precisely we assume that the rate of convergence of the phase function depends on the temperature via an increasing continuous function, and it depends on the phase via a maximal antimonotone graph. We prove the existence and uniqueness of the solution of the relaxed problem, and we perform the asymptotic analysis toward the Stefan problem as the relaxation parameter approaches zero. 
\end{abstract}


\maketitle


\thispagestyle{empty}


\section{Introduction}

The \emph{Stefan problem} is a well-known system of partial differential equations formulated in order to model phase-transition phenomena in a substance attaining two phases (e.g. solid and liquid). It is assumed that the substance is contained in a space domain $\Omega$ during the time interval $\clint{0,T}$, and the energy balance equation reads
\begin{equation}\label{enbal-intro}
  \frac{\partial}{\partial t}(\theta+\chi) - \Delta\theta = g \qquad 
  \text{in $Q := \Omega \times \clint{0,T}$},
\end{equation}
where for simplicity we have normalized to $1$ all the physical constants. The unknowns 
$\theta = \theta(x,t)$ and $\chi = \chi(x,t)$, $(x,t) \in Q$, represent respectively the temperature and the phase function, i.e. the proportion of phases: for instance, if the substance is solid at $x$ at the time $t$, then it
turns out to be $\chi(x,t) = -1$, while $\chi(x,t) = 1$  accounts for a pure liquid state. Intermediate states are represented by the condition $-1 < \chi < 1 $, so that it is allowed the existence of mushy regions where the substance is a mixture of the solid and liquid parts (cf., e.g., \cite[p. 99]{Vis96}). Therefore in order to describe the evolution of the system, the following condition relating $\theta$ and $\chi$ is introduced:
\begin{equation}\label{Stef cond}
  \chi \in \sign(\theta) \qquad \text{in $Q$},
\end{equation}
where $\sign$ denotes the multivalued sign graph (i.e. $\sign(r) := -1$ if $r<0,$ $\sign(r) := [-1,1]$ if $r=0,$ 
$\sign(r):=1$ if $r>0$).
Formula \eqref{Stef cond} is called \textsl{equilibrium condition of Stefan type}, so that $\theta=0$ is the equilibrium temperature at which the two phases can coexist. 
Therefore problem \eqref{enbal-intro}-\eqref{Stef cond} is usually called \textsl{Stefan problem}. 

Of course the Stefan condition \eqref{Stef cond} could also be written in the equivalent form 
\begin{equation}\label{Stef cond b}
  \sign^{-1}(\chi) \ni \theta \qquad \text{in $Q$},
\end{equation}
$\sign^{-1}$ being the inverse relation of the multivalued $\sign$ graph ($\sign^{-1}(r) := 0$ if 
$r \in \opint{-1,1}$, $\sign^{-1}(-1) := \cldxint{-\infty,0}$, $\sign^{-1}(1) := \clsxint{0,\infty}$).

If dynamic supercooling or superheating effects are to be taken into account, then condition 
\eqref{Stef cond b} has to be replaced by a dynamic relation, and usually the following \textsl{relaxation dynamics} for the phase variable 
$\chi$ is used (cf., e.g., \cite{Vis85,Vis96} and their references):
\begin{equation}\label{Vis rel}
  \eps\frac{\partial \chi}{\partial t} + \sign^{-1}(\chi)\ni\theta \qquad \text{in $Q$},
\end{equation}
$\eps$ being a small kinetic positive parameter. It is clear that from a mathematical point of view, the relaxation dynamics can also be modeled by the inclusion
\begin{equation}\label{new rel}
  \eps\frac{\partial \chi}{\partial t} + \chi\in\sign(\theta) \qquad \text{in $Q$},
\end{equation}
but this relation is not equivalent to \eqref{Vis rel}.

The Stefan problem \eqref{enbal-intro}-\eqref{Stef cond} and the Stefan problems with phase relaxation
\eqref{enbal-intro}, \eqref{Vis rel} or \eqref{enbal-intro}, \eqref{new rel} have been extensively studied in the literature.
Concerning \eqref{enbal-intro}--\eqref{Stef cond}, we refer, e.g., to \cite{Dam77,DamKenSat94,Vis96}, instead one can see
\cite{Dam77,Vis85, ColGra93,DamKenSat94,Vis96} for the relaxed problem \eqref{enbal-intro}, \eqref{Vis rel}. In particular in \cite{Vis85}, uniqueness and existence of \eqref{enbal-intro}, \eqref{Vis rel}, coupled with suitable initial-boundary conditions, are proved in the framework 
of Sobolev spaces, and the solution of the relaxed problem is shown to converge to the solution of the problem \eqref{enbal-intro}--\eqref{Stef cond} as $\eps\searrow0$, in a suitable Sobolev topology.
The alternative has been studied in 
\cite{Vis01, Rec02a}, in particular in
\cite{Rec02a} existence, uniqueness, and asymptotic analysis toward
the Stefan problem are proved within the Sobolev setting. Let us also observe that the Stefan problem with phase relaxation can also be studied taking into account a hyperbolic energy balance yielding a finite speed of propagation for the temperature field (see, e.g., \cite{Vis85,ShoWal87,Sh0Wal91,ColRec02,Rec02b,Rec04}).

The natural analytic models \eqref{enbal-intro}, \eqref{Vis rel} and \eqref{enbal-intro}, \eqref{new rel} have some modelling drawbacks. Indeed, as observed in \cite{Vis01}, in \eqref{Vis rel} the rate of the phase $\chi$ does not depend on $\chi$, because the term $\sign^{-1}(\chi)$ only represents a constraint for the phase function, and in \eqref{new rel} the phase depends only on the sign of the temperature $\theta$. It would be more realistic instead that the rate of $\chi$ decays as $\chi$ approaches $1$, and that it also decays as $\theta$ tends to $0$.

In order to overcome this modelling issue it would be convenient to introduce some more general dependencies on $\theta$ and $\chi$ in the phase dynamics. A first proposal in this direction can be found in 
\cite{Vis01}, where it is introduced the phase relaxation 
\begin{equation}\label{phi-rel-mod}
  \eps\frac{\partial \chi}{\partial t} = \psi(\theta,\chi) \qquad \textrm{in }Q
\end{equation} 
for a suitable class of Lipschitz continuous functions $\psi:\erre^2\rightarrow\erre$ which are increasing in 
$\theta$, decreasing in $\chi$, and such that $\psi(\theta,\chi) = 0$ if and only if 
$\chi \in \sign(\theta)$. 

Model \eqref{phi-rel-mod} has some very good features, but 
it requires a very strong regularity of the dependence on the temperature and the phase, namely a Lipschitz continuity in the pair $(\theta,\chi)$: neither the monotone case, nor the multivalued one are allowed.

Therefore we propose a model which takes a step further in the direction of \eqref{phi-rel-mod}. 
To be more precise we consider the model of phase relaxation
\begin{alignat}{3}
 & \eps\frac{\partial\chi}{\partial t} + \beta(\chi)\ni  \gamma(\theta) & \qquad & 
     \text{in $ \Omega\times\opint{0,T}\, $}, \label{newphrel-intro}
\end{alignat}
where $\gamma$ and $\beta$ are increasing continuous and $\beta$ is possibly multivalued function. Of course \eqref{newphrel-intro} is not a proper generalization of \eqref{phi-rel-mod}, but it takes into account much less regular phase relaxation dynamics, and it aims to solve the above mentioned  modelling drawbacks.
Let us note that model \eqref{newphrel-intro} has been considered in the recent paper \cite{Rou23} in the particular case $\beta = \sign^{-1}$, whose only purpose is to provide a constraint for the phase function.

The aim of our present paper is to prove the existence and uniqueness of solutions of \eqref{enbal-intro}, \eqref{newphrel-intro}, and to perform the asymptotic analysis as $\eps$ approaches zero of the model of phase relaxation \eqref{enbal-intro}, \eqref{newphrel-intro}, in order to understand if the solutions of 
\eqref{enbal-intro}, \eqref{newphrel-intro} converge to the solution of the Stefan problem
\begin{alignat}{3}
  & \frac{\partial}{\partial t}(\theta+\chi) - \Delta\theta = g & \qquad &  \text{in $Q$}, \label{Stefan-intro} \\
  & \chi \in \alpha(\theta) & \qquad & \text{in $Q$}, \label{alfacond-intro}
\end{alignat}
where $\alpha$ is an increasing multivalued function which generalizes the $\sign$ function. Of course $\alpha$ cannot be completely arbitrary, but it requires some compatibility conditions with the phase dynamics
\eqref{newphrel-intro}, thus in order to perform our asymptotic analysis we will assume that 
\[
\gamma(\theta) \in \beta(\chi) \quad \Longleftrightarrow \quad \chi \in \alpha(\theta).
\]
The generality of \eqref{newphrel-intro} makes the use of $L^2$-techniques not particularly useful in order to get enough a priori estimates for the asymptotic analysis as $\eps \searrow 0$. Therefore  we need to exploit $L^1$-techniques, and our
analysis is strongly inspired on some $L^1$-estimates performed by Visintin in \cite{Vis01} (see also \cite{MagVerVis89, Vis02}), but we have to adapt those arguments to our nonregular multivalued case. Moreover, in \cite{Vis01} the convergence of the relaxed problem \eqref{enbal-intro}, \eqref{phi-rel-mod} to a suitable weak formulation of the Stefan problem is proved along a subsequence of $\eps$ only, because it is not known if this rather weak form of the Stefan problem has a unique solution.  Instead in our analysis, by assuming the natural physical assumption that $\alpha$ is bounded, we are able to prove the uniform boundedness of the phase function, and this fact allows us to get some $L^2$-estimates which make the solution of \eqref{enbal-intro}, \eqref{newphrel-intro} converge to the  stronger standard formulation of the Stefan problem, having a unique solution. As a result we are able to prove the convergence along the entire index ~$\eps$.

The phase relaxation \eqref{newphrel-intro} with $\gamma$ multivalued would allow to treat with a single model the two phase relaxations \eqref{Vis rel} and \eqref{new rel}, or even more general relaxation dynamics, but this analysis is considerably more difficult due the lack of uniqueness of the relaxed solution. We will investigate this problem in the future paper \cite{RecZan27}.

We supply the systems \eqref{enbal-intro}, \eqref{Stef cond} and \eqref{enbal-intro}, \eqref{newphrel-intro}
with the rather general initial\--\-boun\-dary conditions described as follows: letting $\{\Gamma_0,\Gamma_1\}$ be a partition of the boun\-dary of $\Omega$ into two measurable sets, we take
\begin{alignat}{3}
  & \theta = \theta_D & \qquad & \text{on $\Gamma_0 \times \clint{0,T}$}, \label{b+i cond1-intro} \\
  & \partial_{\textbf{n}}\theta = -\theta_N & \qquad & \text{on $\Gamma_1 \times \clint{0,T}$}, 
      \label{b+i cond2-intro} \\
  & \theta(\cdot,0) + \chi(\cdot,0)  = \theta_0 + \chi_0 & \qquad &  \text{in $\Omega$}, 
      \label{b+i cond3-intro} 
\end{alignat}
where $\theta_D$, $\theta_N$, $\theta_0$, $\chi_0$ are given functions and $\textbf{n}$ is the outward unit vector normal to the boundary of $\Omega$. We assume that $\theta_D$ is a sufficiently smooth function defined on the cylinder $Q$, that $\theta_N : \Gamma_1 \times \clint{0,T} \longrightarrow \erre$ is regular enough, and that there is a  function $u : Q \longrightarrow \erre$ such that $u = \Delta u$ in $Q$, 
$u = \theta_D$ on $\Gamma_0 \times \clint{0,T}$, and  $-\partial_{\textbf{n}} u = \theta_N$ on 
$\Gamma_1 \times \clint{0,T}$ and we set $\tilde{\theta}_0 := \theta_0 - u(\cdot,0)$. Hence we rewrite all the equations in the new unknown $\tilde{\theta} := \theta - u$ so that problem
\eqref{enbal-intro}, \eqref{newphrel-intro}, \eqref{b+i cond1-intro}-\eqref{b+i cond3-intro} reads, writing again 
$\theta$ instead of $\tilde{\theta}$ for simplicity,
\begin{alignat}{3}
  & \frac{\partial}{\partial t}(\theta+\chi) - \Delta\theta = 
     g - \frac{\partial u}{\partial t} + \Delta u &\qquad& \text{in $Q$}, \label{our system1-intro} \\
  & \eps\frac{\partial \chi}{\partial t} +\beta(\chi) \ni \gamma(\theta+u) & \qquad& \text{in $Q$},
      \label{our system2-intro} \\
  & \theta = 0 & \qquad & \text{on $\Gamma_0 \times \clint{0,T}$,} \label{our system3a-intro} \\ 
 &  \partial_{\textbf{n}}\theta=0 & \qquad & \text{on $\Gamma_1 \times \clint{0,T}$},
      \label{our system3b-intro}\\
 & \theta(\cdot,0) + \chi(\cdot,0)  = \theta_0 + \chi_0 & \qquad &  \text{in $\Omega$}.
  \label{our system4-intro}
\end{alignat}
This formulation has the advantage that the boundary conditions for $\theta$ are homogeneus and the wider generality is incorporated in the $u$-terms in the right-hand side of the balance equation and in the non-linearity $\gamma$. 

The outline of the paper is the following. In Section \ref{PreNot} we make precise the assumptions on the data and we state our main results. In Section \ref{S:Peps} we analyze the relaxed problem \eqref{enbal-intro}, \eqref{newphrel-intro}. In Section \ref{S:priorest} we perform all the necessary a priori estimates independent of $\eps$, and in the final Section \ref{S:limit} we perform the asymptotic analysis as the relaxation parameter $\eps$ goes to zero.

 
\section{Main results}

\subsection{Preliminaries and notations}\label{PreNot}
If $d \in \en$, the set of integers greater than or equal to $1$, then the $d$-dimensional Lebesgue measure of an open set $D \subseteq \erre^d$ will be denoted by $|D|$. In the following the locutions ``almost every" and ``almost everywhere" (``a.e.") will always refer to the Lebesgue measure.
Given $p \in \clsxint{1,\infty}$ and a real Banach space $B$, then $\L^p(D; B)$ will denote the space of $p$-integrable $B$-valued functions on $D$; the vector space of essentially bounded $B$-valued functions on $D$ is denoted by $\L^\infty(D; B)$. These spaces will be endowed with their natural norms defined by $\norma{v}{L^p(D;B)} := \left(\int_D\norma{v(x)}{B}^p \de x\right)^{1/p}$ if $p \in \clsxint{1,\infty}$, and by 
$\norma{v}{L^\infty(D;B)} := 
\inf_{w}\sup_{x \in D}\norma{w(x)}{B}$, where the infimum is taken over all bounded Lebesgue-measurable functions $w$ equal to $v$ a.e. in $D$.
If $p=2$ and $B = E$ is a Hilbert space then this norm is induced by the inner product $(v_1,v_2)_{L^2(D;E)} = \int_D (v_1(x), v_2(x))_E \de x$, where 
$(\cdot,\cdot)_E$ is the inner product in $E$. For the theory of integration of vector valued functions we refer, e.g., to \cite[Chapter VI]{Lan93}. We set $L^p(D) := L^p(D;\erre)$ for $p \in \clint{1,\infty}$.

We will use the Sobolev space
$H^1(D)$ $:=$ $\{v \in L^2(D)\ :\ \partial_i v \in L^2(D),\ i = 1,\ldots, d\}$, where $\partial_i v$ denotes the partial derivative of $v$ with respect to the $i$-th variable in the sense of distributions (cf., e.g., \cite{Ada75}). The symbols $\nabla$, $\div$, and $\Delta$ denote respectively the distributional gradient, divergence, and laplacian operators. $H^1(D)$ is a real Hilbert space if it is endowed with the inner product
$(v_1,v_2)_{H^1(D)} := (v_1,v_2)_{L^2(D)} + (\nabla v_1, \nabla v_2)_{L^2(D;\erre^d)}$, 
$ v_1, v_2 \in H^{1}(D)$,
which induces the usual norm $\norm{\cdot}{H^1(D)}$. If $\partial D$ is of Lipschitz class, if $\Gamma_0$ is open in $\partial D$, then the restriction operator 
$C^\infty(\overline{D}) \function C(\Gamma_0) : v \longmapsto v|_{\Gamma_0}$ can be uniquely continuously extended to the linear continuous operator $\gamma_{\Gamma_0} : H^1(D) \function L^2(\Gamma_0)$ 
where $\Gamma_0$ is endowed with the $(n-1)$-dimensional surface (Hausdorff) measure (see, e.g., \cite{LioMag72, Gri85}). The notation $v|_{\Gamma_0} := \gamma_{\Gamma_0}(v)$ is commonly used for  functions $v \in H^1(D)$. 

If $a, b \in \erre$, $a < b$, we set $L^p(a,b;B) := L^p(\opint{a,b};B)$ for $p \in \clint{1,\infty}$ we recall that $W^{1,p}(a,b;B) := \{f \in L^p(a,b;B)\ :\ f' \in L^p(a,b;E)\}$, where $g'$ denotes
the distributional derivative of a function $g : \opint{a,b} \function B$ and, if $B = E$ is a Hilbert space, we define $H^1(a,b;E) := W^{1,2}(a,b;E)$. For the main properties of the Sobolev space $W^{1,p}(\clint{a,b};B)$ and the space of functions with bounded variation $BV(a,b;B)$ we refer, e.g., to \cite[Appendix]{Bre73}. 
It is also convenient to recall the precise definition of maximal monotonicity (cf. \cite{Bre73}).

\begin{Def}
If $\parti{\erre}$ denotes the power set of $\erre$, then a ``multivalued function'' 
$\zeta : \erre \rightarrow \parti{\erre}$ is said to be \emph{maximal monotone} if, setting
$D(\zeta) := \{r \in \erre\ :\ \zeta(r) \neq \vuoto\}$, the following two conditions hold:
\begin{align}
   & (\zeta_{r_1} - \zeta_{r_2})(r_1 - r_2) \ge 0 \qquad 
      \forall r_1, r_2 \in D(\zeta),\ \forall \zeta_{r_1} \in \zeta(r_1),\ \forall \zeta_{r_2} \in \zeta(r_2); \\
   & (\sigma - \zeta_r)(\rho - r) \ge 0, \quad r \in D(\zeta), \quad \zeta_r \in \zeta(r) \quad 
      \Longrightarrow \quad \sigma \in \zeta(\rho).
\end{align}
For every $r \in D(\zeta)$ we denote by $\zeta^o(r)$ the number in $\zeta(r)$ with minimal absolute value, i.e. the unique number such that
\begin{equation}\label{zeta^o}
  |\zeta^o(r)| = \min\{|s|\ :\  s \in \zeta(r)\}.
\end{equation}
The mapping $\zeta^{-1} : \erre \function \parti{\erre}$ is defined by the relation 
\[
r \in \zeta^{-1}(s) \Longleftrightarrow s \in \zeta(r)
\]
and it is maximal monotone if $\zeta$ is maximal monotone.
\end{Def}
For a maximal monotone $\zeta : \erre \rightarrow \parti{\erre}$ we will adopt a slight abuse of notation by writing 
$\zeta(v)$ instead of $\zeta \circ v$ for any $v : \Omega\times[0,T] \function \erre $. Thus, for instance, the formula $w \in \zeta(v)$ a.e. in $Q$ means $w(x,t) \in \zeta(v(x,t))$ for a.e. $(x,t)\in\Omega\times[0,T]$.

\subsection{Main results}
Now we can present our set of assumptions.

\begin{Ass}\label{H} 
The following conditions will be used in the paper.
\begin{itemize}
\item[(A1)] 
  $\Omega \subseteq \erre^n$ is a bounded open connected set with  
  Lipschitz boundary $\Gamma := \partial \Omega$. $\Gamma_0$ and $\Gamma_1$ are open subsets of 
  $\Gamma$ such that 
  $\Gamma_0 \cap \Gamma_1 = \vuoto$. If $\overline\Gamma_0$ and $\overline\Gamma_1$ denote the       
  closures of $\Gamma_0$ and $\Gamma_1$ in $\Gamma$, then we assume that 
  $\overline\Gamma_0 \cup \overline\Gamma_1 = \Gamma$ and that
  $\overline\Gamma_0 \cap \overline\Gamma_1$ is of Lipschitz class. We define
  \begin{align}
& H := L^2(\Omega), \\
& V := H^1_{\Gamma_0}(\Omega) :=  \{v \in \H1\ :\ v|_{\Gamma_0} = 0\},
\end{align}
where $v \longmapsto v|_{\Gamma_0}$ denotes the trace operator on $\Gamma_0$ defined on $\H1$.
The space $V$ is endowed with the inner product induced by $\H1$. The duality between $V$ and its topological dual space $V'$ is denoted by $\fdual{\cdot}{\cdot}$, and if 
we identify $H$ with its dual space, then we have $V \subset H \subset V'$ with dense and compact embeddings,
and
$\fdual{e}{v} = (e,v)_H$ for every $e \in H$,  $v \in V$.

We define the linear continuous operator $A : V \function V'$ by
\begin{equation}\label{operator A}
  \fdual {A v_1}{v_2} := \intom \nabla v_1 \cdot \nabla v_2, \qquad v_1,v_2 \in V.
\end{equation}
The final time of the evolution will be denoted by $T > 0$ and we set $Q := \Omega \times \opint{0,T}$.
\item[(A2)]
 $\beta : \erre \rightarrow \parti{\erre}$ is a maximal monotone mapping such that $D(\beta) \neq \varnothing$.
\item[(A3)]
 $\gamma : \erre \rightarrow \erre$ is continuous, increasing (in the sense that $(\gamma(r_1) - \gamma(r_2))(r_1 - r_2) \ge 0$ for every $r_1, r_2 \in \erre$), and sublinear,
 i.e., there exists $C_\gamma > 0$ such that 
 \begin{equation}\label{gamma sublin}
   |\gamma(r)| \le C_\gamma(1 + |r|) \qquad  \forall r \in \erre.
 \end{equation}
 \item[(A4)] 
  $\alpha : \erre \rightarrow \parti{\erre}$ is maximal monotone, and it is ``bounded'', i.e. there is a constant $M > 0$ such that 
  \begin{equation}
  |s| \le M \qquad  \forall r \in D(\alpha), \quad \forall s \in \alpha(r).
  \end{equation}
\item[(A5)] We assume the following ``compatibility'' conditions between $\alpha$, $\beta$ and $\gamma$:
\begin{align}
   & \alpha(r) \subseteq D(\beta) \quad \forall r \in D(\alpha), \label{Hyp gamma-beta-alfa0} \\
   & \gamma(r) \in \beta(s) \ \Longleftrightarrow \ 
  s \in \alpha(r).\label{Hyp gamma-beta-alfa}
\end{align}

\end{itemize}
\end{Ass}

We will prove the following existence and uniqueness result for a doubly nonlinear relaxed Stefan problem.

\begin{Thm}\label{ex-Peps}
Let \emph{(A1)-(A3)} of Assumptions \ref{H} be satisfied. Assume that $\eps > 0$,  
$f_\eps \in L^1(0,T;H) + L^2(0,T;V')$, $u_\eps \in L^2(Q)$, $\theta_{0\eps} \in L^2(\Omega)$, 
$\chi_{0\eps} \in L^2(\Omega)$, and $\chi_{0\eps}(x) \in \overline{D(\beta)}$ for a.e. $x \in \Omega$. Then there exists a unique pair $(\thetaep,\chiep)$ such that 
\begin{alignat}{3}
  & \thetaep \in L^2(0,T;V) \cap C(\clint{0,T};H), \label{thep reg}\\
  & \thetaep' \in L^1(0,T;H) + L^2(0,T;V'), \label{thep'reg}\\
  & \chiep \in  H^1(0,T;H), \label{chep reg}\\
  & \thetaep' + \chiep' + A\thetaep = f_\eps
  	& \qquad & \text{in $V'$, a.e. in $\opint{0,T}$}, \label{eneqep} \\
  & \eps\chiep' + \beta\big(\chiep\big) \ni \gamma\big(\thetaep + u_\eps\big)
  	& \qquad & \text{a.e. in $Q$}, \label{phreleps in} \\
  &  \thetaep(0) = \theta_{0\eps}, \quad \chiep(0) = \chi_{0\eps} 
  	& \qquad & \text{a.e. in $\Omega$}. \label{chi-i.c.ep}
\end{alignat}
\end{Thm}

We are interested to perform an asymptotic analysis of the solution of Problem \eqref{eneqep}-\eqref{chi-i.c.ep} as $\eps \searrow 0$. We will prove that this solution converges to the solution of the Stefan problem, whose formulation we recall in the following result (see, e.g., \cite{ColGra93, DamKenSat94,Vis96}).

\begin{Thm}
Let \emph{(A1)} and \emph{(A4)} of Assumptions \ref{H} be satisfied. Assume that $f$ $\in$ $L^2(0,T;V')+L^1(0,T;H)$, $\theta_0 \in H$, $\chi_0 \in H$, $u\in L^2(Q)$, $\chi_0 \in \alpha(\theta_0 + u(0))$ a.e. in $\Omega$, and
\begin{equation}\label{condonal}
  \exists s_0 \in  D(\alpha^{-1})\quad : \quad \int_{s_0}^{\chi_0} (\alpha^{-1})^o(s) \de s \in L^1(\Omega),
\end{equation}
where $(\alpha^{-1})^o(s)$ is defined according to \eqref{zeta^o}.
Then there exists a unique pair $(\theta,\chi)$ such that
\begin{alignat}{3}
  &\theta\in L^2(0,T;V) \cap H^1(0,T;V'), \label{wStef th} \\
  & \chi\in L^\infty(Q), \label{wStef ch} \\
  & \theta + \chi \in H^1(0,T;V') \label{wStef th+ch} \\
  & (\theta + \chi)'(t) + A\theta(t) = f(t)  & \qquad &     
      \text{in $V'$, for a.e. $t \in \opint{0,T}$,} \label{wStef pb 1} \\ 
  & \chi \in \alpha\big(\theta + u\big) & \qquad & \text{a.e. in $Q$},  
       \label{wStef pb 2} \\
  &(\theta+\chi)(0)=\theta_0+\chi_0 & \qquad& \text{in $V'$}. \label{wStef pb 3}
\end{alignat}
\end{Thm}

Here is our main convergence theorem, where the symbols $\convergedeb$ and $\convergedebstar$ denote respectively the weak convergence and the weak-star convergence.

\begin{Thm}\label{mainthm}
Let Assumptions \ref{H} be satisfied. Assume that $f \in L^2(0,T;H)$, $u\in W^{1,1}(0,T;L^1(\Omega))$, $\Delta u \in L^1(Q)$, $\theta_0 \in H$, $\chi_0 \in L^\infty(\Omega)$,  $\chi_0 \in \alpha(\theta_0 + u(0))$ a.e. in $\Omega$, and \eqref{condonal} holds. For every $\eps > 0$ assume that  $f_\eps \in L^2(0,T;H)\cap BV(\clint{0,T};L^1(\Omega))$, $u_\eps \in C(\clint{0,T};H^1(\Omega))$, $u_\eps', \Delta u_\eps \in BV(\clint{0,T};L^1(\Omega))$, 
$u_\eps(t)|_\Gamma = u_\eps(0)|_\Gamma$ for a.e. $t\in \clint{0,T}$, $\thetazep \in V$, $\Delta\thetazep \in L^1(\Omega)$, 
$\chi_{0\eps} \in L^\infty(\Omega)$, and $\chi_{0\eps}(x) \in \overline{D(\beta)}$ for a.e. $x \in \Omega$.
Let $(\thetaep,\chiep)$ be the solution of  \eqref{eneqep}-\eqref{chi-i.c.ep}, and $(\theta,\chi)$ be the solution \eqref{wStef th}-\eqref{wStef pb 3}.
If
 \begin{align}
& \sup_{\eps > 0}(\norma{f_\eps}{BV(\clint{0,T};L^1(\Omega))} + \norma{u'_\eps}{BV(\clint{0,T};L^1(\Omega))} + \norma{\Delta u_\eps}{BV(\clint{0,T};L^1(\Omega))})< \infty, \\ 
& \sup_{\eps > 0}\norma{u_\eps}{L^2(Q)}< \infty,\\
&   \gamma(\thetazep(x)+u_\eps(x,0)) - \beta(\chizep(x)) \subseteq \clint{-\eps,\eps} \qquad \text{for a.e. $x \in \Omega$},
\end{align} 
and if
\begin{alignat}{3}
& f_\eps \to f & \quad & \text{in $L^2(0,T;H)$}, \label{f->} \\
& u_\eps \to u & \qquad & \text{in $W^{1,1}(0,T;L^1(\Omega))$}, \\
& \Delta u_\eps \to \Delta u & \quad & \text{in $L^1(0,T;L^1(\Omega))$},  \label{u->}\\
 &\thetazep \to \theta_0,
 & \qquad & \text{in $L^2(\Omega)$}, \label{ic->} \\
 &\chizep \to \chi_0
 & \qquad & \text{in $L^\infty(\Omega)$}, \label{ic->2} 
\end{alignat}
then (along the whole set of indexes $\eps$)
\begin{alignat}{3}
  & \thetaep \to \theta & \quad & \text{in $L^1(Q)$}, 
  \\
  & \thetaep \convergedeb \theta & \quad & \text{in $L^2(Q)$}, \\
  & \chiep \to \chi & \quad & \text{in $L^1(Q)$}, \\
  & \chiep \convergedebstar \chi & \quad & \text{in $L^\infty(Q)$}.
\end{alignat}
\end{Thm}


\section{Analysis of the relaxed Stefan problem}\label{S:Peps}

We start by stating a regularity result for the heat equation in its weak formulation. 

\begin{Lem}\label{L:regolarita' eq. calore}
Let \emph{(A1)} of Assumptions \ref{H} be satisfied. If $F \in L^2(0,T;H)$ and $w_0 \in V$, then there exists a unique function $w$ such that
\begin{alignat}{3}
  & w \in L^\infty(0,T;V) \cap H^1(0 ,T;H), &\ A w \in L^2(0,T;H), \label{reg1} \\
  & A w = - \Delta w, &  \label{reg2}\\	
  & w' -\Delta w = F & \qquad & \text{a.e. in}\ Q, \label{reg3}\\
  & w(0) = w_0 & \qquad & \text{a.e. in}\ \Omega. \label{reg4}
\end{alignat}
\end{Lem}

\begin{proof}
Conditions \eqref{reg1}, \eqref{reg4}, and the fact that $w' + Aw = F$ in $V'$, a.e. in $\opint{0,T}$, follow from a standard regularity result for parabolic equations. Let us now recall the following generalized Green formula holding for every $v \in H^{1}(\Omega)$, $\vett v \in L^2_\div(\Omega)^d = \{\vett v \in L^2(\Omega)\ :\ \div \vett v \in L^2(\Omega)\}$ (cf. \cite[p. 239-240]{DutLio88}):
\[
  \int_\Omega \vett v \cdot \nabla v = - \int_\Omega v \,\mathrm{div} {\vett v} + 
  {_{H^{-1/2}(\Gamma)}\langle \gamma_{\vett n} \mathbf{v},\gamma_0v\rangle_{H^{1/2}(\Gamma)}},
\]
where $\gamma_0 : H^1(\Omega) \function H^{1/2}(\Gamma)$ and $\gamma_{\vett n} : \L^2_\div(\Omega) \function H^{-1/2}(\Gamma)$ are the unique linear continuous surjective (trace) operators such that 
$\gamma_{0}(v) = v|_{\Gamma}$ for every $v \in C^{\infty}(\overline{\Omega})\cap H^1(\Omega)$ and
$\gamma_{\vett n}(\vett w) = \vett w|_{\Gamma} \cdot \mathbf{n}$ for every $w \in C^{\infty}(\overline{\Omega})^d\cap L^2_\div(\Omega)$ (see \cite{Ada75} for the definition of fractional Sobolev spaces).
For all $\varphi \in C_c^\infty(\Omega)$ we have therefore
\begin{align}
  \distrdual{\Delta w}{\varphi} 
    & = \distrdual{w}{\Delta \varphi} = \intom w \Delta \varphi 
	    \notag \\
	& = - \intom \nabla w \cdot \nabla \varphi + 
		{_{H^{-1/2}(\Gamma)}\langle \gamma_{\vett n} \nabla \varphi,\gamma_0w\rangle_{H^{1/2}(\Gamma)}} 
		\notag \\ 
    & = - \intom \nabla w \cdot \nabla \varphi =
	    - \fdual{Aw}{\varphi} = - (Aw,\varphi)_H = - \intom (Aw) \varphi,   \notag
\end{align}
a.e. in $\opint{0,T}$,
whence $-Aw$ is equal to the distributional laplacian $\Delta w$, and it follows that \eqref{reg2}, \eqref{reg3} hold.
\end{proof}

Now we can prove that problem \eqref{eneqep}--\eqref{chi-i.c.ep} has at most one solution.
We will use the auxiliary functions $\sign_\mu : \erre \function \erre$ for $\mu > 0$ and $\sign_0 : \erre \function \erre$ defined by
\begin{align}
& \sign_\mu(r) = \min\{\max\{r/\mu, -1\}, 1\} \qquad \text{for $\mu > 0$}, \label{sgnm}\\
& \sign_0(r) = \lim_{\mu \to 0} \sign_\mu(r). \label{sgn0}
\end{align}

\begin{Prop}\label{P:uniqeps}
Under the assumptions of Theorem \ref{ex-Peps}, there exists at most one pair 
$(\thetaep,\chiep)$ such that 
\eqref{thep reg}-\eqref{chi-i.c.ep} hold.
\end{Prop}

\begin{proof}
Let $(\theta_i,\chi_i)$ such that 
\eqref{thep reg}--\eqref{chi-i.c.ep} hold with $(\thetaep,\chiep)$ replaced by $(\theta_i,\chi_i)$, $i=1,2$.
Let us set $\theta := \theta_1 - \theta_2$, $\chi := \chi_1 - \chi_2$ and observe that by virtue of Lemma 
\ref{L:regolarita' eq. calore} we have $\theta \in L^\infty(0,T;V) \cap H^1(0,T;H)$, $A\theta = -\Delta\theta \in L^2(Q)$, and 
\begin{alignat}{3}
  & \theta'(t) - \Delta\theta(t) = - \chi'(t) & \qquad & \text{a.e. in $\Omega$}, \label{diff enbal}\\
  & \theta(0) = 0 & \qquad & \text{a.e. in $\Omega$}, \label{th ic-i}
\end{alignat}
for a.e. $t \in \opint{0,T}$. We also have that there exists $\xi_i \in L^2(Q)$ such that 
\begin{alignat}{3}
  & \eps\chi_i'(x,t) + \xi_i(x,t) =
  \gamma\big(\theta_i(x,t) + u_\eps(x,t)\big)
  & \qquad & \text{for a.e. $x \in \Omega$}, \label{phreleps-i} \\
  & \xi_i(x,t) \in \beta\big(\chi_i(x,t)\big) &\qquad & \text{for a.e. $x \in \Omega$}, \\
  &  \chi_i(0) = \chi_{0\eps} & \qquad & \text{a.e. in $\Omega$}, \label{chi-i.c.ep-i}
\end{alignat}
for a.e. $t \in \opint{0,T}$, $i=1,2$. Now we fix $t \in \opint{0,T}$ such that \eqref{diff enbal}, \eqref{phreleps-i} hold and such that $(\partial|\theta|/\partial t)|(x,t)$ and$(\partial|\chi|/\partial t)(x,t)$ exist for a.e. $x \in \Omega$: the set of such $t$'s has full measure in $\opint{0,T}$ by virtue of \cite[Theorem 2.1.11, p. 48]{Zie89}.
If $\sign_\mu : \erre \function \erre$ is defined by \eqref{sgnm} for $\mu > 0$, then
$\sign_\mu(\theta(t)) \in V$ and 
\[
  -\intom \Delta\theta(t)\sign_{\mu}(\theta(t)) = 
  \intom |\nabla\theta(t)|^2\sign_{\mu}'(\theta(t)) \ge 0 \qquad \text{for a.e. $t \in \opint{0,T}$},
\]
therefore if we multiply \eqref{diff enbal} by $\sign_{\mu}(\theta(t))$ and integrate over $\Omega$, we obtain
\begin{equation}\label{est_sign-i}
  \intom \theta'(t)\sign_{\mu}(\theta(t)) +
  \intom \chi'(t)\sign_{\mu}(\theta(t)) \le 0.
\end{equation}
Taking the limit as $\mu \searrow 0$ in \eqref{est_sign-i} we infer that 
\begin{equation}\label{est_sign-i2}
  \intom \theta'(t)\sign_{0}(\theta(t)) +
  \intom \chi'(t)\sign_{0}(\theta(t)) \le 0,
\end{equation}
where $\sign_0$ is defined by \eqref{sgn0}.

Now fix $x \in \Omega$ such that $(\partial\chi/\partial t)(x,t)$ and $(\partial|\chi|/\partial t)(x,t)$ exist:
if $\chi(x,t) \neq 0$, then using \eqref{phreleps-i} and the monotonicity of $\beta$ and $\gamma$, we infer that
\begin{align}
  & \chi'(x,t)\sign_{0}(\theta(x,t)) \notag \\
  & = \frac{1}{\eps}\Big[\big(\gamma(\theta_1(x,t)) - \gamma(\theta_2(x,t)\big) - \big(\xi_1(x,t) - \xi_2(x,t)\big) \Big]  
           \sign_{0}(\theta_1(x,t) - \theta_2(x,t)) \notag \\
  & \ge \frac{1}{\eps}\Big[\big(\gamma(\theta_1(x,t)) - \gamma(\theta_2(x,t)\big) - \big(\xi_1(x,t) - \xi_2(x,t)\big) \Big]  \sign_{0}(\chi_1(x,t) - \chi_2(x,t)) \notag \\
  & = \chi'(x,t)\sign_{0}(\chi(x,t)) = \frac{\partial|\chi|}{\partial t}(x,t);\label{monarg1}
\end{align}
if instead $\chi(x,t) = 0$, then it is easy to check that $(\partial|\chi|/\partial t)(x,t) = (\partial\chi/\partial t)(x,t) =  0$ and we have
\begin{equation}
  \chi'(x,t)\sign_{0}(\theta(x,t)) = 0 = \frac{\partial|\chi|}{\partial t}(x,t);
\end{equation}
thus in every case we have
\begin{equation}
  \chi'(x,t)\sign_{0}(\theta(x,t)) \ge \frac{\partial|\chi|}{\partial t} (x,t).\label{monarg3}
\end{equation}
From \eqref{est_sign-i2} and \eqref{monarg3} we therefore get 
\begin{equation}
  \frac{\de}{\de t} \intom(|\theta(t)| + |\chi(t)|) \le 0 \qquad \text{for a.e. $t \in \opint{0,T}$},
\end{equation}
and integrating in time, thanks to \eqref{th ic-i}, \eqref{chi-i.c.ep-i}, we obtain
\begin{equation}
  \intom(|\theta(t)| +|\chi(t)|) \le 0 \qquad \forall t \in \clint{0,T}.
\end{equation}
Therefore
$
  \norm{\theta}{L^\infty(0,T;L^1(\Omega))} +   \norm{\chi}{L^\infty(0,T;L^1(\Omega))} \le 0$,
so that $\theta_1 = \theta_2$, $\chi_1 = \chi_2$ a.e. in $Q$ and we are done.
\end{proof}

We now turn to the question about existence of solutions. We first approximate problem \eqref{eneqep}-\eqref{chi-i.c.ep}.

\begin{Lem}\label{L:lambda approx}
Let \emph{(A1)-(A3}) of Assumptions \ref{H} be satisfied.
Fix $\eps > 0$, $\lambda > 0$, and let $\gamma_\lambda := (\lambda + \gamma^{-1})^{-1} : \erre \function \erre$ be the $\lambda$-Yosida approximation of $\gamma$ (cf. \cite[Section II.4, p. 28]{Bre73}).
Assume that $f_\eps \in L^1(0,T;H) + L^2(0,T;V')$, $u_\eps \in L^2(Q)$, $\thetazep \in L^2(\Omega)$, $\chizep \in L^2(\Omega)$, and 
$\chizep(x) \in \overline{D(\beta)}$ for a.e. $x \in \Omega$. Then there exist a unique pair $(\theta_\lambda, \chi_\lambda) \in [L^2(0,T;V) \cap C(\clint{0,T};H)] \times H^1(0,T;H)$ such that 
$\theta'_\lambda \in L^1(0,T;H) + L^2(0,T;V')$ and 
\begin{alignat}{3}
  & \theta_\lambda' + \chi_\lambda' + A\theta_\lambda = f_\eps 
  	& \qquad & \text{in $V'$, a.e. in $\opint{0,T}$}, \label{eneqmu} \\
  & \eps\chi_\lambda' + \beta\big(\chi_\lambda\big) \ni \gamma_\lambda\big(\theta_\lambda + u_\eps\big) 
  	& \qquad & \text{a.e. in $Q$}, \label{phrelmu} \\
  &  \theta_\lambda(0) = \thetazep, \quad \chi_\lambda(0) = \chizep 
  	& \qquad & \text{a.e. in $\Omega$}, \label{ch-i.c.mu}
\end{alignat}
\end{Lem}

\begin{proof}
Fix $\Theta \in L^2(0,T;H)$ and recall that $\gamma_\lambda$ is $1/\lambda$-Lipschitz continuous over $\erre$. Therefore by \cite[Theorem 3.4, Theorem 3.6, Example 2.3.3]{Bre73} (applied with the maximal monotone operator $\mathcal{B}$ defined 
$
  \mathcal{B}(v) :=  \{w \in L^2(\Omega)\ :\ w(x) \in \beta(v(x)) \quad \text{for a.e. $x \in \Omega$}\}
$)
there exists a unique $\chi \in H^1(0,T;H)$ such that 
\begin{alignat}{3}
  & \eps\chi' + \beta(\chi) \ni \gamma_\lambda(\Theta + u_\eps) & \qquad & \text{a.e. in $Q$}, \label{phrel-mu} \\
  & \chi(0) = \chizep & \qquad & \text{a.e. in $\Omega$}, \label{ch-i.c.-mu}
\end{alignat}
thus (cf., e.g., \cite[Theorem 3.2]{{Bai67}}) there exists a unique $\theta \in L^2(0,T;V) \cap C(\clint{0,T};H)$ such that $\theta' \in L^1(0,T;H) + L^2(0,T;V')$ and 
\begin{alignat}{3}
  & \theta' + A\theta = f_\eps - \chi' & \qquad & \text{in $V'$, a.e. in $\opint{0,T}$}, \label{eneq-mu} \\
  & \theta(0) = \thetazep & \qquad & \textrm{a.e. in $\Omega$}. \label{th-i.c.-mu}
\end{alignat}
In this way we have defined a mapping $\mathcal{S}_\lambda :  L^2(0,T;H) \function L^2(0,T;H)$ associating
$\Theta$ with the unique $\theta$ satisfying \eqref{phrel-mu}, \eqref{ch-i.c.-mu}, \eqref{eneq-mu}, \eqref{th-i.c.-mu}. For every $\Theta_k \in L^2(0,T;H)$, $k = 1, 2$, set $\theta_k := \mathcal{S}_\lambda(\Theta_k)$, 
$\Thetatilde := \Theta_1 - \Theta_2$, and $\thetatilde := \theta_1 - \theta_2$. Moreover set 
$\chitilde := \chi_1 - \chi_2$ where $\chi_k$ is the unique function satisfying \eqref{phrel-mu}, \eqref{ch-i.c.-mu} with $\Theta$ replaced by $\Theta_k$, $k = 1, 2$.
Let us fix $t \in \cldxint{0,T}$.
Let us take the difference of the approximated phase relaxations \eqref{phrel-mu} with $\Theta$ and $\chi$ replaced respectively by $\Theta_k$ and $\chi_k$, , $k = 1, 2$, multiply it by $\chitilde$, and integrate over $\Omega\times\opint{0,t}$. Thanks to 
\eqref{ch-i.c.-mu}, to the maximal monotonicity of $\beta$, to the Lipschitz continuity of $\gamma_\lambda$, and to the Young inequality, we obtain
\begin{align}
  \frac{\eps}{2}\normaq{\chitilde(t)}{H} \le  
  \frac{1}{2\lambda^2}\normaq{\Thetatilde}{L^2(0,t;H)} + \frac{1}{2}\integr\normaq{\chitilde(s)}{H} \de s,
  \label{stcontr1}
\end{align}
therefore by the Gronwall lemma we infer that there exists a constant $C > 0$ independent of $\theta$ (and depending only on $T$, $\eps$ and $\lambda$), such that 
\begin{equation}\label{stcontr1-2}
  \normaq{\chitilde(t)}{H} \le C\normaq{\Thetatilde}{L^2(0,t;H)}.
\end{equation}
Now let us take the difference of the energy balance equations \eqref{eneq-mu} with $\theta$ and $\chi$ replaced respectively by $\theta_k$ and $\chi_k$, $k = 1, 2$. Integrate in time the resulting equation, and multiply it by $\thetatilde$. Integrating over 
$\opint{0,t} \times \Omega$ and using the Young inequality, thanks to \eqref{th-i.c.-mu} we infer that
\begin{align}
  \frac{1}{2}\normaq{\thetatilde}{L^2(0,t;H)} 
  \le 
   \frac{1}{2}\integr\normaq{\chitilde(s)}{H} \de s, \label{stcontr2}
\end{align}
therefore using \eqref{stcontr1-2} we infer that
\begin{equation}
  \normaq{\thetatilde}{L^2(0,t;H)} \le 
   C\integr\normaq{\Thetatilde}{L^2(0,s;H)} \de s, \label{stcontr3}
\end{equation}
hence the mapping $\mathcal{S}_\lambda^n$ is a strict contraction for $n$ sufficiently large, and we can infer 
that $\mathcal{S}_\lambda$ admits a uni\-que fixed point $\theta_\lambda$. We conclude by observing that $(\theta_\lambda, \chi_\lambda)$ is a solution of \eqref{eneqmu}--\eqref{ch-i.c.mu} if and only if $\theta_\lambda$ is a fixed point of $\mathcal{S}_\lambda$.
\end{proof}

\begin{proof}[Proof of Theorem \ref{ex-Peps}]
By Proposition \ref{P:uniqeps} we only have to prove existence of solutions. 
For simplicity we omit the subscripts $\eps$. 
For any $\lambda > 0$ let $\gamma_\lambda := (\lambda + \gamma^{-1})^{-1} : \erre \function \erre$ be the $\lambda$-Yosida approximation of $\gamma$ (cf. \cite[Section II.4, p.28]{Bre73}). By Lemma \ref{L:lambda approx} there exists a unique pair 
$(\theta_\lambda, \chi_\lambda) \in [L^2(0,T;V) \cap C(\clint{0,T};H)] \times H^1(0,T;H)$ such that 
$\theta'_\lambda \in L^1(0,T;H) + L^2(0,T;V')$ and 
\eqref{eneqmu}--\eqref{ch-i.c.mu} hold. Let us test the energy balance equation \eqref{eneqmu} by $\theta_\lambda$ and integrate over $\opint{0,t}\times \Omega$. Recalling that $f = f_{V'} + f_H$ with $f_{V'} \in L^2(0,T;V')$ and $f_{H} \in L^1(0,T;H)$, and exploiting the Young inequality, we find 
\begin{align}
  & \frac{1}{2}\normaq{\theta_\lambda(t)}{H} + 
      \frac{1}{2}\normaq{\nabla\theta_\lambda}{L^2(0,t;H^d)} \notag \\
  &  \le 
      \frac{1}{2}\normaq{\thetazep}{H} + \frac{1}{2}\normaq{f_{V'}}{L^2(0,T;V')} \notag \\
  & \phantom{=\ } + 
      \integr\norma{f_{H}(s)}{H}\norma{\theta_\lambda(s)}{H}\de s + 
      \frac{\eps+2}{2\eps}\integr\normaq{\theta_\lambda(s)}{H}\de s +
      \frac{\eps}{4}\normaq{\chi'_\lambda}{L^2(0,t;H)}. \label{en_est_lamu}
\end{align}
Let $\xi_\lambda \in L^2(0,T;H)$ be such that
\begin{alignat}{3}
  & \xi_\lambda \in \beta(\chi_\lambda) & \quad & \text{a.e. in $Q$}, \label{phrel-xil0}\\
  & \eps\chi_\lambda' + \xi_\lambda = \gamma_\lambda(\theta_\lambda + u)& \quad & \text{a.e. in $Q$}.
  \label{phrel-xil}
\end{alignat}
Let us recall that $|\gamma_\lambda(r)| \le |\gamma(r)|$ 
for every $r \in \erre$ (cf. e.g., \cite[Proposition 2.6]{Bre73})  so that from  \eqref{gamma sublin} we get that
\begin{equation}\label{gammal sl}
  |\gamma_\lambda(r)| \le C_\gamma(1 + |r|) \qquad \forall r \in \erre.
\end{equation} 
Therefore if we multiply the approximated phase relaxation \eqref{phrel-xil} 
by $\chi'_\lambda$ and integrate over 
$\opint{0,t}\times\Omega$, taking some $r_0 \in D(\beta)$ and recalling definition \eqref{zeta^o},
we get
\begin{align}
  & \eps\normaq{\chi'_\lambda}{L^2(0,t;H)} + 
      \int_0^t\intom (\xi_\lambda(s)) - \beta^o(r_0))\chi'_\lambda(s) \de s \notag \\
  & =  -\int_0^t\intom \beta^o(r_0)\chi'_\lambda(s) \de s +
          \integr\intom \gamma_\lambda(\theta_\lambda(s) + u(s))\chi'_\lambda(s) \de s \notag \\
  & \le \int_0^t\intom |\beta^o(r_0)||\chi'_\lambda(s)| \de s + 
          C_\gamma\integr\intom(1 + |\theta_\lambda(s) + u(s)|)|\chi'_\lambda(s)|\de s \notag \\
  & \le \frac{\eps}{4}\normaq{\chi'_\lambda}{L^2(0,t;H)} +
          \frac{2}{\eps}t|\Omega||\beta^0(r_0)|^2 + 
          \frac{4C_\gamma^2}{\eps}\integr\intom(1 + 2|\theta_\lambda(s)|^2 + 2|u(s)|^2)\de s.  
          \label{ph_est_lamu} 
\end{align}
Let $\phi_\beta : \erre \function \cldxint{0,\infty}$ be the convex lower semicontinuous function such that $\partial\phi_\beta = \beta$ and $\phi_{\beta_\mu}(r_0) = 0$ (cf., e.g., \cite{Bre73} for the definition of the subdifferential $\partial\phi_\beta$), so that 
\[
  \phi_{\beta}(\chi_\lambda(t)) - \beta^o(r_0)(\xi_\lambda(t)-r_0) \ge 0 \quad \text{for a.e. $t\in [0,T]$,}
\]
then by virtue of \cite[Lemma 3.3, p. 73]{Bre73}
\[
  \frac{\de\ }{\de t}[\phi_{\beta}(\chi_\lambda(t)) - \beta^o(r_0)(\chi_\lambda(t)-r_0)] =
  [\xi_\lambda(t) - \beta^o(r_0)]\chi_\lambda'(t),
\]
thus
\[
   \intom \phi_{\beta}(\chi_\lambda(t)) - \beta^o(r_0)(\xi_\lambda(t)-r_0) =  \int_0^t\intom (\xi_\lambda(s)) - \beta^o(r_0))\chi'_\lambda(s) \de s \ge 0
\]
for a.e. $t \in \clint{0,T}$, and from \eqref{ph_est_lamu} we infer that 
\begin{align}
  \frac{3\eps}{4}\normaq{\chi_\lambda'}{L^2(0,t;H)} 
       \le \frac{2t|\Omega|(|\beta^o(r_0)|^2+2C_\gamma^2)}{\eps} +
           \frac{8C_\gamma^2}{\eps}\normaq{u}{L^2(0,t;H)} +
           \frac{8C_\gamma^2}{\eps}\integr\normaq{\theta_\lambda(s)}{H}\de s.
           \label{ph_est_lamu2}
\end{align}
Thus adding \eqref{en_est_lamu} to \eqref{ph_est_lamu2} we deduce that
\begin{align}
  & \frac{\eps}{2}\normaq{\chi_\lambda'}{L^2(0,t;H)} +
      \frac{1}{2}\normaq{\theta_\lambda(t)}{H} + 
      \frac{1}{2}\normaq{\nabla\theta_\lambda}{L^2(0,t;H^d)} \notag \\
&  \le C_1(\eps)\left(
       1 + \integr \norma{f_{H}(s)}{H}\norma{\theta_\lambda(s)}{H}\de s + 
      \integr\normaq{\theta_\lambda(s)}{H}\de s\right), \label{ph_est_lamu3}
\end{align}
where $C_1(\eps) > 0$ is a constant independent of $\lambda$, 
but depending only on $\eps$, $T$, $|\Omega|$, $C_\gamma$, $\beta, r_0$, $u$.
Using a generalized version of the Gronwall lemma (see, e.g., \cite[Theorem 2.1]{Bai67}), by \eqref{ph_est_lamu3} and by a comparison in \eqref{eneqmu}, we find a constant $C(\eps) > 0$ independent of $\lambda$, but depending only on $\eps$, $T$, $|\Omega|$, $C_\gamma$, $\beta$, $u$, such that 
\begin{equation}\label{bound for lamu}
   \norma{\theta_\lambda}{L^2(0,T;V)} + 
   \norma{\theta_\lambda'}{L^2(0,T;V')} +
    \norma{A\theta_\lambda}{L^2(0,T;V')} +
    \norma{\chi_\lambda}{H^1(0,T;H)} \le C(\eps).
\end{equation}
Therefore there exists $\chi \in H^1(0,T;H)$ and 
$\theta \in L^1(0,T;V)\cap L^2(0,T;H)$ such that 
$\theta' \in L^2(0,T;V')$ and, at least for a subsequence, the following weak convergences hold:
\begin{alignat}{3}
  & \chi_\lambda \convergedeb \chi & \qquad &  \text{in $H^1(0,T;H)$},
\label{deb-chilambda} \\
  & \theta_\lambda \convergedeb \theta & \qquad & \text{in $L^2(0,T;V)$}, \label{thetla->thet-1}\\
  & \theta_\lambda' \convergedeb \theta', \ 
  A\theta_\lambda \convergedeb A\theta
  & \qquad & \text{in $L^2(0,T;V')$}. \label{thetla->thet-2}
\end{alignat}
By the Aubin-Lions compactness lemma 
(\cite[Theorem 5.1, p. 58]{Lio69}),
\begin{equation}\label{thetla->thet}
  \theta_\lambda \to \theta \qquad \text{in $L^2(0,T;H)$},
\end{equation}
thus taking possibly a further subsequence 
\begin{equation}\label{thl -> qo}
  \theta_\lambda(x,t) \to \theta(x,t) \qquad \text{for a.e. $(x,t) \in Q$}.
\end{equation}
We claim that
\begin{equation}
  \gamma_\lambda(\theta_\lambda(x,t) + u(x,t)) \to \gamma(\theta(x,t) + u(x,t)) \qquad \text{for a.e. $(x,t) \in Q$}. \label{etala --> eta}
\end{equation}
Indeed, fix $(x,t)$ such that \eqref{thl -> qo} holds, by \eqref{gammal sl}, passing possibly to a further subsequence, we have that there exists $\eta(x,t) := \lim_{\lambda \searrow 0} \gamma_\lambda(\theta_\lambda(x,t) + u(x,t))$; moreover, using e.g. \cite[Proposition 2.6-(iii), p. 28]{Bre73}, we have $\lim_{\lambda \searrow 0}\gamma_\lambda(r) = \gamma^o(r) = \gamma(r)$ for every $r \in \erre$, thus 
\begin{align}
  & [\eta(x,t) - \gamma(r)][(\theta(x,t) + u(x,t)) - r] \notag \\
  & \lim_{\lambda \searrow 0} [\gamma_\lambda(\theta_\lambda(x,t) + u(x,t)) - \gamma_\lambda(r)][(\theta_\lambda(x,t) + u(x,t)) - r] \ge 0, \label{lim gamla(thla)}
\end{align}
for every $r \in \erre$, hence \eqref{lim gamla(thla)} and the monotonicity of $\gamma$ imply that $\eta(x,t) = \gamma(\theta(x,t) + u(x,t))$. Claim \eqref{etala --> eta} is therefore proved, and by the dominated convergence theorem we also have 
\begin{equation}
  \gamma_\lambda(\theta_\lambda + u) \to \gamma(\theta + u) \qquad \text{in $L^2(0,T;H)$}. \label{etal --> eta0}
\end{equation}

We take the limit in the energy balance equation \eqref{en_est_lamu} and find that
\begin{equation}\label{en bal limit}
  \theta' + \chi' + A\theta = f \qquad \text{in $V'$, a.e. in $\opint{0,T}$}.
\end{equation}
therefore from \eqref{bound for lamu} we infer that there exist $\eta \in L^2(0,T;H)$ such that, at least for a subsequence,
\begin{equation}
  \gamma_\lambda(\theta_\lambda + u) \to \gamma(\theta + u) \qquad \text{in $L^2(0,T;H)$}, \label{etal --> eta}
\end{equation}
Now let $\xi_\lambda \in L^2(0,T;H)$ be such that
\eqref{phrel-xil0}-\eqref{phrel-xil} hold.
A comparison in \eqref{phrel-xil}, together with \eqref{bound for lamu}, \eqref{etal --> eta}, yields a constant $C_1 > 0$ independent of $\lambda$ such that
\begin{equation}
   \norma{\xi_\lambda}{L^2(0,T;H)} \le C_1.
\end{equation}
Thus there exists $\xi \in L^2(0,T;H)$ such that, at least for a subsequence, we have the following weak convergence:
\begin{equation}
  \xi_\lambda \convergedeb \xi \qquad \text{in $L^2(0,T;H)$}, \label{xil-->xi}
\end{equation}
and taking the limit in \eqref{phrel-xil} we infer that 
\begin{equation}
  \eps\chi' + \xi = \gamma(\theta + u) \qquad \text{a.e. in $Q$}. \label{pheqlalim}
\end{equation}
We are left to pass to the limit in the nonlinear operator $\beta$. First observe that from \eqref{deb-chilambda} it is easy to check that
\begin{equation}
\chi_\lambda(t) \convergedeb \chi(t) \qquad \text{in $H,\quad  \forall t \in \clint{0,T}$}. \label{chil(t)->deb}
\end{equation}
Thus exploiting \eqref{phrel-xil}, \eqref{etal --> eta} \eqref{deb-chilambda}, \eqref {chil(t)->deb} and  \eqref{pheqlalim} we infer that
\begin{align}
  \limsup_{\lambda \searrow 0}(\xi_\lambda,\chi_\lambda)_{L^2(Q)}
 & = \limsup_{\lambda \searrow 0} (\gamma_\lambda(\theta_\lambda + u) - \eps\chi_\lambda',\chi_\lambda)_{L^2(Q)} \notag \\
 & = \limsup_{\lambda \searrow 0} \left[(\gamma_\lambda(\theta_\lambda + u),\chi_\lambda)_{L^2(Q)} - \frac{\eps}{2}\norma{\chi_\lambda(T)}{H} + \frac{\eps}{2}\norma{\chizep}{H}  \right] \notag \\
 & \le (\gamma(\theta + u),\chi)_{L^2(Q)} - \liminf_{\lambda \searrow 0}\frac{\eps}{2}\norma{\chi_\lambda(T)}{H} +
 \frac{\eps}{2}\norma{\chizep}{H}  \notag \\
 & \le (\gamma(\theta + u),\chi)_{L^2(Q)} - \frac{\eps}{2}\norma{\chi(T)}{H} + 
 \frac{\eps}{2}\norma{\chizep}{H}  \notag \\
 & = (\gamma(\theta + u) - \eps\chi',\chi)_{L^2(Q)} = 
 (\xi,\chi)_{L^2(Q)},
\end{align}
thus by \cite[Proposition 2.5, p. 27]{Bre73} we infer that $\xi \in \beta(\chi)$ a.e.\ in $Q$
and we are done.
\end{proof}


\section{A priori estimates}\label{S:priorest}

\subsection{Uniform boundedness of the phase}

In this section we prove that the phase variable $\chi_\eps$ of the relaxed system is uniformly bounded if it is bounded at the initial time. Before we need the following lemma.

\begin{Lem}\label{L:inf beta-gamma--inf alfa}
Let \emph{(A2)-(A5)} of Assumptions \ref{H} be satisfied, and let $\theta\in\erre$, $\chi\in D(\beta)$. Then
\begin{align}
  & \xi > \sup\alpha(\theta) \quad \Longrightarrow \quad \sup[\gamma(\theta) - \beta(\chi)] < 0,
  \label{lemmino1} \\
    & \chi < \inf\alpha(\theta) \quad \Longrightarrow \quad \inf[\gamma(\theta) - \beta(\chi)] > 0.
  \label{lemmino2}
\end{align}
\end{Lem}

\begin{proof}
We first prove \eqref{lemmino1}. Assume therefore that $\chi > \sup\alpha(\theta)$, pick $\widehat\chi \in D(\beta)$ such that
\begin{equation}
  \widehat\chi \in \alpha(\theta).
\end{equation}
By \eqref{Hyp gamma-beta-alfa0} and 
\eqref{Hyp gamma-beta-alfa} we have that
\begin{equation}
  \gamma(\theta) \in \beta(\widehat\chi)
\end{equation}
and
\begin{equation}\label{a-dh c<b-2}
  \inf(\gamma(\theta) - \beta(\widehat\chi)) = \gamma(\theta)-\sup\beta(\widehat\chi)<0<
  \sup(\gamma(\theta) - \beta(\widehat\chi)) = \gamma(\theta)-\inf\beta(\widehat \chi).
\end{equation}
Since $\chi > \sup\alpha(\theta)$, we have that $\widehat\chi < \chi$, so that, by the monotonicity of $\beta$, 
\begin{equation}\label{betaxih<betaxi}
   \inf\beta(\widehat\chi) \le \sup\beta(\widehat\chi) \le \inf\beta(\chi) \le \sup\beta(\chi),
\end{equation}
and from assumption \eqref{Hyp gamma-beta-alfa} we infer that $\gamma(\theta) \not\in  \beta(\chi)$, i.e. either 
$0 < \inf(\gamma(\theta) - \beta(\chi))$, or $\sup(\gamma(\theta) - \beta(\chi)) < 0$. Thus the proof of \eqref{lemmino1} is completed if we  prove that
$0 \ge \inf(\gamma(\theta) - \beta(\chi))$: and in order to do this, we assume by contradiction that $0 < \inf(\gamma(\theta) - \beta(\chi)) = \gamma(\theta) - \sup\beta(\chi)$. This condition, together with \eqref{betaxih<betaxi}, yields
\[
   \sup\beta(\widehat\chi) < \gamma(\theta),
\]
contradicting \eqref{a-dh c<b-2}. The proof of \eqref{lemmino2} is analogous.
\end{proof}

Now we can prove the uniform boundedness of the phase variable by extending the argument of \cite[Lemma 4.2]{Rec23} to the nonregular multivalued case.

\begin{Prop}\label{|ph|< 1}
Let Assumptions \ref{H} be satisfied.
If $\eps > 0$, let $a_\eps, b_\eps \in \erre$, $\chiep \in H^1(0,T;H)$, $\theta_\eps\in L^2(Q)$, and $u_\eps \in L^2(Q)$ be such that
\begin{alignat}{3}
  & \eps\chiep'(x,t) \in \gamma\big(\thetaep(x,t) + u_\eps(x,t)\big)  - \beta\big(\chiep(x,t)\big) & 
     \qquad & \text{for a.e. $(x,t) \in Q$}, \label{ph-rel-eq-|ph|< 1} \\
  & \chiep(x,0) \in \clint{a_\eps,b_\eps}  & \quad & \text{for a.e. $x \in \Omega$}, \\
  & \alpha(r) \subseteq \clint{a_\eps,b_\eps} & \quad & \forall r \in D(\alpha). \label{R(alfa) in a,b}
\end{alignat}
Then
\begin{equation}
  a_\eps \le \chi_{\eps}(x,t) \le b_\eps \qquad  \text{for a.e. $(x,t) \in Q$}.
\end{equation}
\end{Prop}

\begin{proof}
Let us observe that $\chiep' \in L^1(Q)$, therefore by the Fubini theorem we have that $\chiep'(x,\cdot) \in L^1(0,T)$ for a.e.\ $x \in \Omega$. If $\varphi \in C_c^\infty(\opint{0,T})$ and $z \in L^2(\Omega)$, then Fubini theorem also entails that
\begin{align}
  & \int_\Omega z(x)\int_0^T(\chiep(x,t)\varphi'(t)+\chiep'(x,t)\varphi(t))\de t \de x  \notag \\
  & = \int_0^T\int_\Omega z(x)(\chiep(x,t)\varphi'(t)+\chiep'(x,t)\varphi(t))\de x \de t  \notag \\
  & = \int_0^T(z,\chiep(t))_H \varphi'(t)\de t + \int_0^T(z,\chiep'(t))_H \varphi(t) \de t.  \notag
\end{align}   
From the previous chain of equalities, recalling that $\chiep \in H^1(0,T;H)$ and using again the Fubini theorem, we infer that       
\begin{align}
&  \int_\Omega z(x)\int_0^T(\chiep(x,t)\varphi'(t)+\chiep'(x,t)\varphi(t))\de t \de x  \notag \\
& = \int_0^T\Big(z,\chiep(0) + \int_0^t\chiep'(s) \de s\Big)_H \varphi'(t)\de t +
        \int_0^T(z,\chiep'(t))_H \varphi(t) \de t  \notag \\
  & = (z,\chiep(0))_H\int_0^T \varphi'(t)\de t + \int_0^T\int_0^t(z,\chiep'(s))_H \varphi'(t)\de s\de t +
        \int_0^T(z,\chiep'(t))_H \varphi(t) \de t  \notag \\ 
  & =  \int_0^T(z,\chiep'(s))_H\int_s^T\varphi'(t) \de t\de s + \int_0^T(z,\chiep'(t))_H  \varphi(t)\de t  \notag \\
  & =  -\int_0^T(z,\chiep'(s))_H\varphi(s)\de s + \int_0^T(z,\chiep'(t))_H  \varphi(t)\de t  =0,        \notag
\end{align}
whence, by the arbitrariness of $z$, we infer that 
$\int_0^T(\chiep'(x,t)\varphi(t) +\chiep(x,t)\varphi'(t))\de t = 0$ for a.e.\ $x \in \Omega$, i.e.\ that 
$\chiep'(x,\cdot)$ is the distributional derivative of $\chiep(x,\cdot)$ for a.e.\ $x \in \Omega$. Therefore there exists a measurable set $A \subseteq \Omega$ such that $|\Omega \smallsetminus A| = 0$ and
\[
  \chiep(x,t) = \chi_{\eps}(x,0) + \int_0^t \chiep'(x,s) \de s \qquad \forall t \in \clint{0,T}, 
  \qquad \forall x \in A.
\]
It follows that for every $x \in A$ the function $\chiep(x,\cdot) : \clint{0,T} \function \erre$ is absolutely continuous. It is not restrictive to assume that $a_\eps \le \chiep(x,0) \le b_\eps$ for every $x \in A$. 
Let us fix $x \in A$ and prove that $a_\eps \le \chiep(x,t) \le b_\eps$ for every $t \in \clint{0,T}$. Indeed, if by contradiction this were not true, there would exist $t_0 \in \cldxint{0,T}$ such that $\chiep(x,t_0) \not\in \clint{a_\eps,b_\eps}$.  
Let us first assume that $\chiep(x,t_0) > b_\eps$. Then, by continuity, there exists 
$a_0 \in \clsxint{0,t_0}$ such that $\chiep(x,a_0) = b_\eps$ and $\chiep(x,t) > b_\eps$ for every $t \in \cldxint{a_0,t_0}$. Therefore by \eqref{R(alfa) in a,b} we have $\chiep(x,t) > \sup\{\alpha(r)\ :\ r \in D(\alpha)\}$, in particular
$\chiep(x,t) > \sup\{\alpha\big(\thetaep(x,t) + u_\eps(x,t)\big)\}$ for every $t \in \cldxint{a_0,t_0}$. Hence by Lemma 
\ref{L:inf beta-gamma--inf alfa} 
we have that
$\sup[\gamma(\thetaep(x,t) + u_\eps(x,t)) - \beta(\chiep(x,t))] < 0$. It follows that $\chiep'(x,t) < 0$ for a.e. $t \in \cldxint{a_0,t_0}$, therefore, as $\chiep(x,\cdot)$ is absolutely continuous, we infer that $\chiep(x,\cdot)$ is decreasing on $\cldxint{a_0,t_0}$, a contradiction. The remaining case when $\chiep(x,t_0) < a_\eps$ is dealt in a completely analogous way.
\end{proof}

\subsection{$L^2$-estimates}

In the reminder of the paper we will need to use the function 
$\widehat{h} : \clint{0,T} \function V'$ defined by 
\begin{equation}\label{hat h-fine}
    \widehat{h}(t) := \integr h(s) \de s,
\end{equation}
for every $h \in L^1(0,T;V')$.

\begin{Lem}\label{L2estlem}
Let Assumptions \ref{H} be satisfied.
Let $\eps_0 > 0$ be fixed and assume that $0 < \eps \le \eps_0$ and that
$f_\eps \in L^2(0,T;H)$, $u_\eps \in L^2(Q)$, $\theta_{0\eps} \in L^2(\Omega)$, 
$\chi_{0\eps} \in L^2(\Omega)$, and $\chi_{0\eps}(x) \in \overline{D(\gamma)}$ for a.e. $x \in \Omega$. If $(\thetaep,\chiep)$ is the solution of 
\eqref{thep reg}-\eqref{chi-i.c.ep}, and if
\begin{equation}\label{chzep bdd}
 \sup_{0 < \eps \le \eps_0} \norm{\chizep}{L^\infty(\Omega)} < \infty,
\end{equation}
then there exists $C > 0$ independent of $\eps$ such that
\begin{align}
& \norma{\thetaep}{L^2(0,t;H)} +
\norma{\thetahatep}{L^\infty(0,T;V)} \notag \\
& \le C(1 + \normaq{\thetazep}{H} + \normaq{\chizep}{H} + \normaq{\hat{f}_\eps}{L^2(0,t;H)} + \normaq{u_\eps}{L^2(Q)}).\label{L2est}
\end{align}
\end{Lem}

\begin{proof}
Let us first observe that by \eqref{chzep bdd} and by Proposition \ref{|ph|< 1}, $\chiep$ is uniformly bounded w.r.t. $\eps$ by a constant $K > 0$.
Therefore integrating the energy balance equation \eqref{eneqep} in time, testing the resulting equation by $\thetaep$, and integrating in time and space we get
\begin{align}
  & \normaq{\thetaep}{L^2(0,t;H)}+
  \frac{1}{2}\normaq{\nabla\thetahatep(t)}{H^d} \notag \\
  & =
  \integr\intom (\thetazep + \chizep +\hat{f}_\eps(s))\thetaep(s) \de
  s -
  \integr\intom \chiep(s)\thetaep(s) \de s \notag \\
  & \le 2(1 + K)(\normaq{\thetazep}{H} + \normaq{\chizep}{H} + \normaq{\hat{f}_\eps}{L^2(0,t;H)}) + \frac{1}{4}\normaq{\thetaep}{L^2(0,t;H)}. \label{L2-uno}
\end{align}
Now  let $r_0 \in D(\beta)$. From \eqref{phreleps in} we have that there exists $\xiep\in L^2(Q)$ with $\xiep\in\beta(\chiep)$ such that $\eps\chiep'+\xiep - \beta^o(r_0) =\gamma(\thetaep+u_\eps) - \beta^o(r_0)$ a.e. in $Q$, thus
 by multiplying this last equation by $\chiep - r_0$ and observing that $(\xiep(x,t) - \beta^o(r_0))(\chiep(x,t) - r_0) \ge 0$ due to the monotonicity of $\beta$, we infer the existence of a constant $C_1 > 0$ independent of $\eps$, depending only on $C_\gamma$, $\beta^o(r_0)$, $\eps_0$, such that
\begin{align}
\frac{\eps}{2}\normaq{\chiep(t)}{H} 
& \le \integr\intom
[\gamma(\thetaep(s) + u_\eps(s))-\beta^o(r_0)](\chiep(s) - r_0)\de s +
\eps\integr\intom r_0\chiep'(s) \de s \notag \\
& \le \integr\intom [C_\gamma(1 + |\thetaep(s)| + |u_\eps(s)|)+|\beta^o(r_0)|]|\chiep(t)-r_0| + 
\eps r_0\intom(\chiep(t) - \chizep) \notag \\
& \le \frac{1}{4}\normaq{\thetaep}{L^2(0,t;H)} + 
C_1(1 + \normaq{\chizep}{H}+ \normaq{u_\eps}{L^2(Q)} + K). \label{L2-due}
\end{align}
Thus adding the two previous inequalities \eqref{L2-uno} and \eqref{L2-due}, we get a constant $C > 0$, independent of $\eps$, such that \eqref{L2est} holds.
\end{proof}

\subsection{$L^1$-estimates}

In this section we perform the necessary $L^1$-estimates needed in our limit procedure.

\begin{Lem}\label{est1}
Let \emph{(A1)-(A3)} of Assumptions \ref{H} be satisfied. 
For every $\eps > 0$ let us assume that 
\begin{align}
& f_\eps \in L^2(0,T;H) \cap BV(\clint{0,T};L^1(\Omega)), \label{fBV} \\
&  u_\eps \in C(\clint{0,T};H^1(\Omega)), \quad u_\eps', \Delta u_\eps \in BV(\clint{0,T};L^1(\Omega)), \label{DDuBV} \\
& u_\eps(t)|_\Gamma = u_\eps(0)|_\Gamma \qquad  \text{for a.e. $t\in \clint{0,T}$}, \label{u(t)=onbd} \\
& \thetazep \in V, \quad \Delta \thetazep \in L^1(\Omega),\quad \chi_{0\eps} \in L^2(\Omega),
\end{align}
and
\begin{equation}\label{datieps piccoli}
  \gamma(\thetazep(x)+u_\eps(x,0)) - \beta(\chizep(x)) \subseteq \clint{-\eps,\eps} \qquad \text{for a.e. $x \in \Omega$}.
\end{equation}
If $(\thetaep,\chiep)$ is such that 
\eqref{thep reg}-\eqref{chi-i.c.ep} hold,
then $A \thetaep \in L^\infty(0,T;L^1(\Omega))$  and 
there exists a constant $C > 0$ independent of $\eps$ (depending only on $T$ and $|\Omega|$) such that
\begin{align}
  & \norma{\thetaep}{L^\infty(0,T;L^1(\Omega))} + 
  \norma{\chiep}{L^\infty(0,T;L^1(\Omega))} \notag \\
  & \le
  \norma{\thetazep}{L^1(\Omega)} + \norma{\chizep}{L^1(\Omega)} + \norma{u_\eps(0)}{L^1(\Omega)}
  + \norma{\Delta\thetazep}{L^1(\Omega)} + \norma{\Delta u_\eps(0)}{L^1(\Omega)} \notag \\
  & \phantom{\le \ } + C(1 + \norma{f_\eps}{L^1(Q)} + \norma{u_\eps'}{L^1(0,T;L^1(\Omega)(\Omega))} + \norma{\Delta u_\eps}{L^1(0,T;L^1(\Omega))}),
  \label{estla-th-chi-supL1}
\end{align}  
and
\begin{align}\label{e-stima2-Vis01-det}
  & \norma{\thetaep'}{L^\infty(0,T;L^1(\Omega))} + \norma{\chiep'}{L^\infty(0,T;L^1(\Omega))} + 
  \norma{A\thetaep}{L^\infty(0,T;L^1(\Omega))}\notag \\
  & \le C + 
  \norma{\Delta\thetazep}{L^1(\Omega)} + \norma{\Delta u_\eps(0)}{L^1(\Omega)} \notag \\
  & \phantom{\le\ } +
  \norma{f_\eps}{BV(\clint{0,T};L^1(\Omega))} +
  \norma{u'_\eps}{BV(\clint{0,T};L^1(\Omega))} + 
  \norma{\Delta u_\eps}{BV(\clint{0,T};L^1(\Omega))},
\end{align}
where for every $h \in BV(\clint{0,T};L^1(\Omega))$ we adopt the convention that
$h(0) := \lim_{t \searrow 0} h(t)$.
\end{Lem}

\begin{proof}
Let us observe that by Lemma \ref{L:regolarita' eq. calore} we have that $A\thetaep = -\Delta\thetaep \in L^2(0,T;H)$ and  $\thetaep' + \chiep' - \Delta \thetaep =
f_\eps$ a.e. in $Q$, hence if we set 
\begin{align}
  & \thetabarep := \thetaep + u_\eps, \qquad \label{tetalabar}
     \overline\theta_{0\eps} := \thetazep + u_\eps(0)\\
  & \overline{f}_\eps := f_\eps + u_\eps' - \Delta u_\eps, \label{fbar}
\end{align}
we have that
\begin{alignat}{3}
   & \thetabarep'(t) + \chiep'(t) - \Delta\thetabarep(t) = \overline{f}_\eps(t) 
   & \qquad &  \text{a.e. in $\Omega$,} \label{eneqmu-over-det} \\
   & \thetabarep(0) = \overline\theta_{0\eps}& \qquad &  \text{a.e. in $\Omega$},
\end{alignat}	 
for a.e. $t \in \opint{0,T}$. Moreover let $\xiep \in L^2(Q)$ be such that 
\begin{alignat}{3}
  & \eps\chiep'(x,t) + \xiep(x,t) =
  \gamma\big(\thetaep(x,t) + u_\eps(x,t)\big)
  & \qquad & \text{for a.e. $x \in \Omega$}, \label{phreleps} \\
  & \xiep(x,t) \in \beta(\chiep(x,t)) & \qquad & \text{for a.e. $x \in \Omega$}, \label{phreleps2}
\end{alignat}
for a.e. $t \in \opint{0,T}$. Now we fix $t \in \opint{0,T}$ such that \eqref{eneqmu-over-det}, \eqref{phreleps} hold and such that $(\partial|\thetabarep|/\partial t)(x,t)$ and $(\partial|\chiep|/\partial t)(x,t)$ exist for a.e. $x \in \Omega$: the set of such $t$'s has full measure in $\opint{0,T}$ by virtue of \cite[Theorem 2.1.11, p. 48]{Zie89}.
Now we adapt the argument in  \cite[p. 1477]{Vis01} to the multivalued case and to the presence of the datum $u_\eps$. Recalling that $\sign_\mu(r)$ is defined by \eqref{sgnm} for $\mu > 0$, by \eqref{u(t)=onbd} we have that $\thetabarep(t) - \overline\theta_{0\eps}(t) \in V$ and
$\sign_\mu(\thetabarep(t) - \overline\theta_{0\eps}(t)) \in V$, hence
\begin{align}
  & -\intom\Delta\thetabarep(t)\sign_\mu(\thetabarep(t)  - \overline\theta_{0\eps}) \notag \\
  & = -\intom\Delta(\thetabarep(t) - \overline\theta_{0\eps})
  {\sign_\mu(\thetabarep(t)-\overline\theta_{0\eps}) } 
      - \intom\Delta\overline\theta_{0\eps}\sign_\mu(\thetabarep(t)-\overline\theta_{0\eps})\notag \\
  & =  \intom|\nabla(\thetabarep(t) -\overline\theta_{0\eps})|^2
                     \sign_\mu'(\thetabarep(t) -\overline\theta_{0\eps}) - 
         \intom \Delta\overline\theta_{0\eps} \sign_\mu(\thetabarep(t) -\overline\theta_{0\eps} ), \notag \\
  & \ge -\intom\Delta\overline\theta_{0\eps} \sign_\mu(\thetabarep(t) -\overline\theta_{0\eps} )
           \ge -\intom |\Delta\overline\theta_{0\eps}|. \label{e-stima1-1-Vis01}
\end{align}
Thus, multiplying \eqref{eneqmu-over-det} by $\sign_\mu(\thetabarep(t) - \overline\theta_{0\eps})$, we infer that
\begin{align}
  \intom(\thetabarep'(t) + \chiep'(t))\sign_\mu(\thetabarep(t)-\overline\theta_{0\eps})
  \le \intom|\Delta\overline\theta_{0\eps}| + \intom|\overline f_\eps(t)|,
\end{align}
and, taking the limit as $\mu \searrow 0$,
\begin{equation}\label{e-stima1-2-Vis01}
  \intom(\thetabarep'(t) + \chiep'(t))\sign_0(\thetabarep(t) -\overline\theta_{0\eps}) \le 
   \intom|\Delta\overline\theta_{0\eps}| + \intom|\overline f_\eps(t)|,
\end{equation}
where $\sign_0$ is defined by \eqref{sgn0}.
Now fix $x \in \Omega$ such that $(\partial\chiep/\partial t)(x,t)$ and $(\partial|\chiep-\chizep|/\partial t)(x,t)$ exist and 
define $\xizep^t : \Omega \function \erre$ in such a way that 
\[
  \xizep^t(x) \in \beta(\chizep(x))
  \quad \text{and} \quad
  \begin{cases}
    \xiep(x,t) \ge \xizep^t(x) & \text{if $\chiep(x,t) \ge \chizep(x)$}, \\
    \xiep(x,t) \le \xizep^t(x) & \text{if $\chiep(x,t) < \chizep(x)$}.
  \end{cases}
\]
By \eqref{phreleps}, by the monotonicity of $\gamma$ and $\beta$, and by \eqref{datieps piccoli}, we find
\begin{align}
  & \chiep '(x,t)\sign_0(\thetabarep(x,t)  - \overline\theta_{0\eps}(x)) \notag \\
  &   = \frac{1}{\eps}(\gamma(\thetabarep(x,t)) - \xiep(x,t))
        \sign_0(\thetabarep(x,t)-\overline\theta_{0\eps}(x)) \notag \\
  & = \frac{1}{\eps}[(\gamma (\thetabarep(x,t)) - \gamma(\overline\theta_{0\eps}(x)))-(\xiep(x,t) - \xizep^t(x))] 
        \sign_0(\thetabarep(x,t) -\overline\theta_{0\eps}(x)) \notag\\  
  &\phantom{=\ } + \frac{1}{\eps}(\gamma(\overline\theta_{0\eps}(x,t)) - \xizep^t(x))
                              \sign_0(\thetabarep(x,t) - \overline\theta_{0\eps}(x)) \notag \\
  & \ge \frac{1}{\eps}[(\gamma (\thetabarep(x,t)) - \gamma(\overline\theta_{0\eps}(x)))-(\xiep(x,t) - \xizep^t(x))]
                                  \sign_0(\chiep(x,t) - \chizep(x)) \notag \\
  &\phantom{=\ } +  
        \frac{1}{\eps}(\gamma(\overline\theta_{0\eps}(x,t)) - \xizep^t(x))
                              \sign_0(\thetabarep(x,t) - \overline\theta_{0\eps}(x)) \notag \\
  & =  \chiep'\sign_0(\chiep(x,t) - \chizep(x)) \notag \\
  & \phantom{=\ }+ 
         \frac{1}{\eps}(\gamma(\overline\theta_{0\eps}(x))-\xizep^t(x))
         [\sign_0(\thetabarep(x,t)-\overline\theta_{0\eps}(x))-\sign_0(\chiep(x,t) -\chizep(x))] \notag \\
  & \ge \frac{\partial |\chiep - \chizep|}{\partial t}(x,t) - 2  \qquad \text{a.e. in $Q$}. \label{stima1-3-Vis01}
\end{align}
Hence we have proved that 
\begin{equation}
  \frac{\de}{\de t} \intom(|\thetabarep(t) - \overline{\theta}_{0\eps}| + |\chiep(t) - \chizep|) 
  \le 
    \intom(|\Delta\overline\theta_{0\eps}| + |\overline{f}_\eps(t)| + 2)
   \quad \text{for a.e. $t \in \opint{0,T}$}, \notag
\end{equation}
and integrating in time we obtain
\begin{align}
  & \norma{\thetabarep(t) - \overline{\theta}_{0\eps}}{L^1(\Omega)} + 
  \norma{\chiep(t) - \chizep}{L^1(\Omega)} \notag \\
  & \le 
  t(\norma{\Delta\overline\theta_{0\eps}}{L^1(\Omega)} + 2|\Omega|) +
  \integr\norma{\overline{f}_\eps(s)}{L^1(\Omega)}\de s \quad \forall t \in \clint{0,T}.
  \label{st1bar}
\end{align}
Since 
$|\norma{\thetaep(t)-\thetazep}{L^1(\Omega)} - \norma{u_\eps(t)-u_\eps(0)}{L^1(\Omega)}|
   \le \norma{\thetabarep(t) - \overline\theta_{0\eps}}{L^1(\Omega)}$, from \eqref{st1bar}
we infer that

\begin{align}
  & \norma{\thetaep(t)-\thetazep}{L^1(\Omega)} +  \norma{\chiep(t) - \chizep}{L^1(\Omega)} \notag \\
  & \le \norma{u_\eps(t)-u_\eps(0)}{L^1(\Omega)} + 
         t(\norma{\Delta\overline\theta_{0\eps}}{L^1(\Omega)} + 2|\Omega|) +
  \integr\norma{\overline{f}_\eps(s)}{L^1(\Omega)}\de s
  \qquad \forall t \in \clint{0,T},
 \notag\label{e-stima1-Vis01-det}
\end{align}
which in particular also implies \eqref{estla-th-chi-supL1}.
Now for any $h \in \opint{0,T}$ and for any $v:\clint{0,T}\longrightarrow V'$ we define 
$\delta_h v : \opint{0,T-h} \longrightarrow V'$ by
\begin{equation}\label{e-deltah}
   \delta_h v (t)=v(t+h)-v(t), \qquad t \in \opint{0,T-h}.
\end{equation}
By \eqref{eneqmu-over-det} we get
\begin{equation}\label{e-3.16-Vis01}
  (\delta_h \thetabarep)'(t) + (\delta_h\chiep)'(t) - \Delta(\delta_h\thetabarep(t))=\delta_h\overline f_\eps(t) \qquad \text{a.e. in $\Omega$}
\end{equation}
for a.e. $t \in \opint{0,T-h}$. We fix $t \in \opint{0,T-h}$ such that \eqref{e-3.16-Vis01} and  \eqref{phreleps} hold, and such that
$(\partial|\delta_h\thetabarep|/\partial t)(x,t)$ and $(\partial|\delta_h\chiep|/\partial t)$$(x,t)$ exist for a.e. $x \in \Omega$: the set of such $t$'s has full measure in $\opint{0,T}$ by virtue of \cite[Theorem 2.1.11, p. 48]{Zie89}.
We test \eqref{e-3.16-Vis01} by $\sign_\mu(\delta_h\thetabarep) \in V$ and, arguing exactly as in \eqref{e-stima1-1-Vis01}--\eqref{e-stima1-2-Vis01} we get
\begin{equation}\label{e-3.17-Vis01}
  \intom \left((\delta_h\thetabarep)'(t)+(\delta_h\chiep)'(t)\right)\sign_0(\delta_h\thetabarep(t))
   \leq \intom |\delta_h \overline f_\eps(t)| \qquad \text{for a.e. $t \in \opint{0,T-h}$}.
\end{equation}
If we fix $x \in \Omega$ such that $(\partial\delta_h\chiep/\partial t)(x,t)$ and $(\partial|\delta_h\chiep|/\partial t)(x,t)$ exist, 
by the phase relaxation \eqref{phreleps} we have that
\begin{align}
   & (\delta_h\chiep)'(x,t)\sign_0(\delta_h\thetabarep(x,t)) \notag \\
   & = \frac{1}{\eps}(\delta_h\gamma (\thetabarep(x,t)) -\delta_h\xiep(x,t))\sign_0(\delta_h\thetabarep(x,t)), \notag 
\label{e-3.18-Vis01}
\end{align}
and the same monotonicity argument used in \eqref{monarg1}-\eqref{monarg3} yields
\begin{equation}
(\delta_h\chiep)'(x,t)\sign_0(\delta_h\thetabarep(x,t)) \ge 
\frac{\partial}{\partial t}|\delta_h\chiep|(x,t).
\end{equation}
Thus from \eqref{e-3.17-Vis01} we get
\begin{equation*}
\frac{\de}{\de t}  \intom (|\delta_h\thetabarep(t)|+ |\delta_h\chiep(t)|)
   \leq \intom |\delta_h \overline f_\eps(t)|, \qquad \text{for a.e. $t \in \opint{0,T-h}$},
\end{equation*}
hence integrating in time and exploiting 
\eqref{st1bar} we obtain
\begin{align}
   &\norma{\delta_h\thetabarep(t)}{L^1(\Omega)}+ \norma{\delta_h\chiep(t)}{L^1(\Omega)}\notag\\
   &\leq \norma{\delta_h\thetabarep(0)}{L^1(\Omega)}+\norma{\delta_h\chiep(0)}{L^1(\Omega)} + \integr\norma{\delta_h\overline{ f}_\eps(s)}{L^1(\Omega)} \de s \notag \\
   & \le h(\norma{\Delta\overline{\theta}_{0\eps}}{L^1(\Omega)} + 2|\Omega|) +
   \int_0^h\norma{\overline{f}_\eps(s)}{L^1(\Omega)}\de s + \integr\norma{\delta_h\overline{ f}_\eps(s)}{L^1(\Omega)}\de s.\label{est2aux} 
 \end{align}
Since $
   |\norma{\delta_h\thetaep(t)}{L^1(\Omega)}-\norma{\delta_h u_\eps(t)}{L^1(\Omega)}| \le \norma{\delta_h\thetabarep(t)}{L^1(\Omega)}$,
from \eqref{est2aux} we get
 \begin{align}
   &\norma{\delta_h\thetaep(t)}{L^1(\Omega)}+\norma{\delta_h\chiep(t)}{L^1(\Omega)} \notag\\
& \le \norma{\delta_h u_\eps(t)}{L^1(\Omega)} + h(\norma{\Delta\overline{\theta}_{0\eps}}{L^1(\Omega)} + 2|\Omega|) \notag \\
 & \phantom{\le\ }+
   \int_0^h\norma{\overline{f}_\eps(s)}{L^1(\Omega)}\de s + \integr\norma{\delta_h\overline{ f}_\eps(s)}{L^1(\Omega)}\de s\notag\label{e-3.19-Vis01}
\end{align}
for every $t \in \opint{0,T-h}$. Hence taking the limit as $h \searrow 0$, using \eqref{fBV},\eqref{DDuBV} and \cite[Lemma A.1]{Bre73}, and by comparison in \eqref{eneqep}, we infer \eqref{e-stima2-Vis01-det}.
\end{proof}

\begin{Lem}\label{stima3}
Let the assumptions of Lemma \ref{est1} be satisfied.
For every $j \in \en$ set 
\begin{equation}\label{omegaj}
  \Omega_j := \{x \in \Omega\ :\ d(x,\erre^d \setmeno \Omega) > 1/j\}. 
\end{equation}
For $v : Q \function \erre$ and $k \in \erre^d$ define 
$\tau_k v : \Omega_j \times \opint{0,T} \function \erre$ by
\begin{equation}\label{deltah}
  \tau_k v(x,t) := v(x+k,t)-v(x,t), \qquad x \in  \Omega_j,\ |k| < 1/j.
\end{equation}
If $(\thetaep,\chiep)$ is such that 
\eqref{thep reg}-\eqref{chi-i.c.ep} hold,
then
\begin{align}
  \|\tau_k\chiep\|_{L^\infty(0,T;L^1(\Omega_j))} 
  & \le 
   \norma{\tau_k\thetazep}{L^1(\Omega_j)} + \norma{\tau_k\chizep}{L^1(\Omega_j)}  + \norma{\tau_k u_\eps}{L^\infty(0,T;L^1(\Omega_j))} 
     \notag\\
   &\phantom{\le} \ +  
   \left\|\tau_k\big(\fep +u_\eps'-\Delta u_\eps\big)\right\|_{L^1(\Omega_j\times\opint{0,T})}. \label{stimrappincr}
\end{align} 
\end{Lem}

\begin{proof}
Observe that $(\Omega_j)$ is an increasing sequence of open subsets invading $\Omega$. If we define by 
$\thetabarep$, $\overline{\theta}_{0\eps}$ and $\overline{f}_\eps$ by
\eqref{tetalabar}, \eqref{fbar}, then from Lemma \ref{L:regolarita' eq. calore} we get 
\begin{equation}\label{e-3.23-Vis01-det}
  (\tau_k\thetabarep)'(t) + (\tau_k\chiep)'(t)  - \Delta(\tau_k\thetabarep(t)) = \tau_k \overline{f}_\eps(t) \quad 
  \text{a.e. in }\Omega_j,\ |k| < 1/j
\end{equation}
for a.e. $t\in \opint{0,T}$.
We fix $t \in \opint{0,T}$ such that \eqref{e-3.23-Vis01-det} and \eqref{phreleps} hold, and such that
$(\partial|\tau_k\thetabarep|/\partial t)(x,t)$ and $(\partial|\tau_k\chiep|/\partial t)$$(x,t)$ exist for a.e. $x \in \Omega$: the set of such $t$'s has full measure in $\opint{0,T}$ by virtue of \cite[Theorem 2.1.11, p. 48]{Zie89}.
We multiply \eqref{e-3.23-Vis01-det} by $\sign_\mu(\tau_k\thetabarep)(t) \in V$, and arguing exactly as in \eqref{e-stima1-1-Vis01}--\eqref{e-stima1-2-Vis01} we get
  \begin{equation}
  \int_{\Omega_j}((\tau_k\thetabarep)'(t) +(\tau_k\chiep)'(t))\sign_{0}(\tau_k \thetabarep(t))
   \le \int_{\Omega_j}  |\tau_k \overline{f}_\eps(t)| \qquad \text{for a.e. $t \in \opint{0,T}$}. \notag
\end{equation}
If we fix $x \in \Omega$ such that $(\partial\tau_k\chiep/\partial t)(x,t)$ and $(\partial|\tau_k\chiep|/\partial t)(x,t)$ exist, 
by the phase relaxation \eqref{phreleps} we have that
\begin{align}
   & (\tau_k\chiep)'(x,t)\sign_0(\tau_k\thetabarep(x,t)) \notag \\
   & = \frac{1}{\eps}(\tau_k\gamma(\thetabarep)(x,t) -\tau_k\xiep(x,t))\sign_0(\tau_k\thetabarep(x,t)), \notag 
\label{e-3.18-Vis01}
\end{align}
where $\xiep \in L^2(Q)$ is such that \eqref{phreleps}-\eqref{phreleps2} hold.
Thus the monotonicity argument used in \eqref{monarg1}-\eqref{monarg3} yields
\begin{equation}
(\tau_k\chiep)'(x,t)\sign_0(\tau_k\thetabarep(x,t)) \ge 
\frac{\partial}{\partial t}|\tau_k\chiep|(x,t).
\end{equation}
Therefore
\begin{equation*}
\frac{\de}{\de t} \int_{\Omega_j} (|\tau_k\thetabarep(t)|+ |\tau_k\chiep(t)|)
   \leq \int_{\Omega_j} |\tau_k \overline{f}_\eps(t)|, \qquad \text{for a.e. $t \in \opint{0,T}$},
\end{equation*}
and integrating in time we obtain
\begin{align}
& \norma{\tau_k\thetabarep(t)}{L^1(\Omega_j)} + \norma{\tau_k\chiep(t)}{L^1(\Omega_j)}\notag \\
   &\leq \norma{\tau_k\thetabarep(0)}{L^1(\Omega_j)}+ 
   \norma{\tau_k\chiep(0)}{L^1(\Omega_j)} + \integr\norma{\tau_k\overline{f}_\eps(s)}{L^1(\Omega_j)}\de s.\label{est2-aux} 
\end{align}
 Since 
$|\tau_k\thetaep(t)| \le |\tau_k\thetabarep(t)| + |\tau_k u_\eps(t)|$,
from \eqref{est2-aux} we infer that
\begin{align}
   & \norma{\tau_k\thetaep(t)}{L^1(\Omega_j)} + \norma{\tau_k\chiep(t)}{L^1(\Omega_j)}\notag \\
   & \le 
   \norma{\tau_k\thetazep}{L^1(\Omega_j)} + \norma{\tau_k\chizep}{L^1(\Omega_j)} +\norma{\tau_k u_\eps(0)}{L^1(\Omega_j)}  + \norma{\tau_k u_\eps(t)}{L^1(\Omega_j)} 
    \notag\\
   &\phantom{\le} \ + \int_0^T\norma{\tau_k \overline{f}_\eps(s)}{L^1(\Omega_j)}  \de s\qquad \forall t \in (0,T],\label{e-3.25-Vis01-det}
\end{align}
which yields \eqref{stimrappincr}, due to the definition of $ \overline\fep$ (see \eqref{fbar}).
\end{proof}


\section{Asymptotic limit to the Stefan problem}
\label{S:limit}

Before proceeding with the limit procedure, we need the following uniqueness result for a suitable weak formulation of the Stefan problem. 

\begin{Lem}\label{uniqintStef}
Let \emph{(A1)} and \emph{(A4)} of Assumptions \ref{H} be satisfied. If $f \in L^2(0,T;V') + L^1(0,T;H)$, $\chi_0$, $\theta_0\in L^2(\Omega)$, 
then there exists at most one pair $(\theta,\chi) : Q \function \erre^2$ satisfying the following conditions:
\begin{alignat}{3}
  &\theta \in L^2(0,T;H), \label{P pb 0}\\
  & \thetahat\in L^\infty(0,T;V)\cap H^1(0,T;V'), \label{P pb 1} \\
  & \chi\in L^\infty(Q), \\
  & \theta(t) + \chi(t) + A\thetahat(t) = \fhat(t) + \theta_0 + \chi_0 & \qquad & 
      \text{in $V'$, for a.e. $t \in \opint{0,T}$,}
   \label{P pb 2} \\ 
  & \chi(x,t) \in \alpha\big(\theta(x,t) + u(x,t)\big) & \qquad & \text{for a.e. $(x,t)\in Q$}, \label{P pb 3}
\end{alignat}
where $\thetahat$ and $\fhat$ are defined according to \eqref{hat h-fine}.
A pair $(\theta,\chi)$ satisfying \eqref{P pb 1}-\eqref{P pb 3} is also called a 
\textsl{solution of the Stefan problem in the sense of Baiocchi-Duvaut-Fr\'emond}.
\end{Lem}

\begin{proof}
Let 
$(\theta_i,\chi_i),\ i=1,2,$ be two solutions of \eqref{P pb 1}--\eqref{P pb 3}, and set 
\begin{equation}
  \thetadiff:=\theta_1-\theta_2, \qquad \chidiff:=\chi_1-\chi_2.
\end{equation} 
Taking the difference of the equations \eqref{P pb 2}  written for 
$(\theta_1,\chi_1)$ and $(\theta_2,\chi_2),$ we find
\begin{eqnarray}
&\thetadiffhat\in L^{\infty}(0,T;V)\cap H^1(0,T;V'),\quad  \\  
 & \chidiff\in L^\infty(Q),\\
&\thetadiff+\chidiff+A\thetadiffhat=0\quad 
  \textrm{in}\ V',\ \ {\rm in}\ \opint{0,T}.\label{int-eq in H}
\end{eqnarray}
By a comparison in the last equation, we see that 
$A\thetadiffhat\in L^1(0,T;H),$ therefore 
multiplying \eqref{int-eq in H} by $\thetadiff$, and integrating over 
$(0,t)\times\Omega$,
we get 
\begin{equation}
\normaq{\thetadiff}{L^2(0,t;H)}+
\integr\intom\chidiff(x,s)\thetadiff(x,s) \de x \de s +  
\frac{1}{2}\intom|\nabla\thetadiffhat(x,t)|^2 \de x=0.  \label{last eq}
\end{equation}
Therefore, since $\chidiff\thetadiff\geq0$ a.e. in $Q$ by \eqref{P pb 3} and by the monotonicity of $\alpha$, from \eqref{last eq} we infer that $\thetadiff=0$ a.e. in $Q$, and therefore $A\thetadiffhat=0$ a.e. in $Q$. By a comparison in 
\eqref{int-eq in H}, also $\chidiff=0$ a.e. in $Q$, and the uniqueness of problem \eqref{P pb 1}-\eqref{P pb 3} is proved.
\end{proof}

We are now ready to prove our main convergence theorem.

\begin{proof}[Proof of Theorem \ref{mainthm}]
We can assume that $\eps \le 1$.
Integrating in time the energy balance equation \eqref{eneqep} we get 
\begin{equation}
  \thetaep + \chiep + A\thetahatep = \thetazep + \chizep + \widehat{f}_\eps
  \qquad \text{in $V'$, a.e. in $\opint{0,T}$,}\label{int eq eps}
\end{equation}
where $\thetahatep$ and $\widehat{f}_\eps$ are defined  defined according to \eqref{hat h-fine}. Let us observe that thanks to (A4) of Assumptions \ref{H}, and to \eqref{ic->2}, we can apply
Lemma \ref{L2estlem}, therefore  we have that there exist $\theta \in L^2(0,T;H)$, $\chi \in L^\infty(Q)$ such that $\thetahat \in L^\infty(0,T;V)$ and 
\begin{alignat}{3}
  & \thetaep \convergedeb \theta & \quad & \text{in $L^2(0,T;H)$}, \\
  & \thetahatep \convergedebstar \thetahatep & \quad & \text{in $L^\infty(0,T;V) \cap H^1(0,T;H)$}, \\
  & \chiep \convergedebstar \chi & \quad & \text{in $L^\infty(Q)$},
\end{alignat}
therefore we can take the limit in \eqref{int eq eps} and obtain
\begin{equation}
  \theta + \chi + A\thetahat = \theta_0 + \chi_0 + \widehat{f}
  \qquad \text{in $V'$, a.e. in $\opint{0,T}$.}\label{int eq}
\end{equation}
Let $\xiep \in L^2(Q)$ be such that
\begin{align}
& \xiep(x,t) \in \beta(\chiep(x,t)) & \quad &
\text{for a.e. $(x,t) \in Q$}, \label{lim1}\\
& \eps\chiep'(x,t) + \xiep(x,t) = \gamma(\thetaep(x,t) + u_\eps(x,t)) & \quad &
\text{for a.e. $(x,t) \in Q$}.\label{lim2}
\end{align}
From \eqref{estla-th-chi-supL1} of Lemma \ref{est1} we deduce that 
\[
  \norma{\thetaep}{L^1(Q)} + \norma{\chiep}{L^1(Q)} \le C,
\]
for a suitable constant $C > 0$ independent of $\eps$. Moreover, recalling notations \eqref{omegaj}, \eqref{deltah}, from Lemma \ref{stima3} we infer that for every $j$ and every $k$ we have
\begin{align}
  & \norma{\tau_k\thetaep}{L^1(\Omega_j\times\opint{0,T})} +\norma{\tau_k\chi_\eps}{L^1(\Omega_j\times\opint{0,T})} \notag \\
& = \norma{\tau_k\thetaep}{L^1(\opint{0,T};L^1(\Omega_j))} + \norma{\tau_k\chi_\eps}{L^1(\opint{0,T};L^1(\Omega_j))} \notag \\
& \le
   T(\norma{(\tau_k\thetaep)}{L^\infty(0,T;L^1(\Omega_j))} +
   \norma{(\tau_k\chi_\eps)}{L^\infty(0,T;L^1(\Omega_j))} )\notag \\
   &\le  T(\norma{\tau_k\thetazep}{L^1(\Omega_j)} + 
           \norma{\tau_k\chizep}{L^1(\Omega_j)}  +
           \norma{\tau_k u_\eps}{L^\infty(0,T;L^1(\Omega_j))}) \notag\\ 
    &\phantom{\le\ } + T\norma{\tau_k(\fep+ u_\eps'-\Delta u_\eps)}{L^1(\Omega_j\times\opint{0,T})},
   \end{align}
therefore from \eqref{f->}-\eqref{ic->} and from a suitable versione of the Riesz-Frechet-Kolmogorov theorem (cf. \cite[Theorem 2.22]{Ada75}) we infer that, at least for a subsequence which we do not relabel, the following strong convergences hold:
\begin{alignat}{3}
  & \thetaep \to \theta & \quad & \text{in $L^1(Q)$}, \\
  & \chiep \to \chi & \quad & \text{in $L^1(Q)$}.
\end{alignat}
Estimate \eqref{e-stima2-Vis01-det} yields the strong convergence
\begin{equation}
\eps\chiep' \to 0 \quad \text{in $L^1(Q)$}.
\end{equation}
Therefore, taking possibly a further subsequence, we find a subset $\tilde{Q}$ of full measure in $Q$
such that 
\begin{alignat}{3}
   & \thetaep(x,t) \to \theta(x,t), \quad
   u_\eps(x,t) \to u(x,t) & \quad & \text{in $\erre$, $\forall (x,t) \in \tilde{Q}$} \\
   &  \chiep(x,t) \to \chi(x,t), \quad
    & \quad & \text{in $\erre$, $\forall (x,t) \in \tilde{Q}$}, \\
 & 
   \eps\chiep'(x,t) \to 0  & \quad & \text{in $\erre$, $\forall (x,t) \in \tilde{Q}$}. \label{lim7}
\end{alignat}
Hence by the continuity of $\gamma$ and by a comparison in \eqref{lim2} we have that for every $(x,t) \in \tilde{Q}$ there exists $\xi(x,t) \in \erre$ such that
\begin{alignat}{3}
   & \gamma(\thetaep(x,t)+u_\eps(x,t)) \to \gamma(\theta(x,t)+u(x,t))& \quad & \text{in $\erre$, $\forall (x,t) \in \tilde{Q}$} \\
 &  \xiep(x,t) \to \xi(x,t), \quad
   \eps\chiep'(x,t) \to 0  & \quad & \text{in $\erre$, $\forall (x,t) \in \tilde{Q}$}.
\end{alignat}
Therefore, taking the limit as $\eps \searrow  0$ in \eqref{lim2} we obtain
\[
\xi(x,t) = \gamma(\theta(x,t) + u(x,t)).
\]
Let us also observe that from \eqref{lim1}, \eqref{lim7} we obtain (see, e.g., \cite[ Proposition 2.5, p. 27]{Bre73} applied for monotone graphs in $\erre$) that
\[
 \xi(x,t) \in \beta(\chi(x,t)) \quad  \forall (x,t) \in \tilde{Q},
\]
so that 
\[
  \gamma(\theta(x,t) + u(x,t)) \in \beta(\chi(x,t)),
\]
and from the compatibility condition (A5) in Assumptions \ref{H} we get $\chi(x,t) \in \alpha(\theta(x,t) + u(x,t))$.
Hence we have proved that $(\theta,\chi)$ solves \eqref{P pb 0}-\eqref{P pb 3}, which has at most one solution by virtue of Lemma \ref{uniqintStef}. This uniqueness property, together with the fact that the formulation \eqref{wStef th}-\eqref{wStef pb 3} is stronger than the formulation of problem \eqref{P pb 1}-\eqref{P pb 3}, let us infer that $(\theta,\chi)$ is indeed the solution of 
\eqref{wStef th}-\eqref{wStef pb 3}.
By this uniqueness property, we also have that the whole sequence
$(\thetaep,\chiep)$ converges to $(\theta,\chi)$.

\end{proof}


\end{document}